\documentclass{amsart}
\usepackage[usenames]{color}
\usepackage{amsfonts}
\usepackage{amssymb}
\usepackage{graphicx}
\usepackage{amscd}
\usepackage{amsmath}
\usepackage{amssymb}
\usepackage{graphicx,pstricks,pst-plot}
\newtheorem{thm}{theorem}[section]
\newtheorem{theorem}[thm]{Theorem}
\newtheorem{corollary}[thm]{Corollary}
\newtheorem{lemma}[thm]{Lemma}
\newtheorem{proposition}[thm]{Proposition}

\newtheorem{definition}[thm]{Definition}
\newtheorem{remark}[thm]{Remark}

\newtheorem{conjecture}[thm]{Conjecture}

\begin{document}

\title[Graded identities with involution for the algebra of upper triangular matrices]{Graded identities with involution for the algebra of upper triangular matrices}

\thanks{Corresponding Author: D. Diniz}

\author[D. Diniz]{Diogo Diniz}
\address{Unidade Acadêmica de Matemática, Universidade Federal de Campina Grande, Campina Grande, PB, 58429-970, Brazil}
\email{diogo@mat.ufcg.edu.br}
\thanks{D. Diniz was partially supported by CNPq grants No.~303822/2016-3, No.~406401/2016-0 and No.~421129/2018-2.}

\author[A. Ramos]{Alex Ramos}
\address{Universidade de Pernambuco, Caruaru, PE, 55294-902, Brazil}
\email{alex.borges@upe.br}
\thanks{A. Ramos was partially supported by CNPq grant No.~406401/2016-0}

\author[J. L. Galdino]{José Lucas Galdino}
\address{Unidade Acadêmica de Matemática, Universidade Federal de Campina Grande, Campina Grande, PB, 58429-970, Brazil}
\email{joselucasgaldinodasilva1997@hotmail.com}



\begin{abstract}
	Let $F$ be a field of characteristic zero. We prove that if a group grading on $UT_m(F)$ admits a graded involution then this grading is a coarsening of a $\mathbb{Z}^{\lfloor\frac{m}{2}\rfloor}$-grading on $UT_m(F)$ and the graded involution is equivalent to the reflection or symplectic involution on $UT_m(F)$. A finite basis for the $(\mathbb{Z}^{\lfloor\frac{m}{2}\rfloor},\ast)$-identities is exhibited for the reflection and symplectic involutions and the asymptotic growth of the $(\mathbb{Z}^{\lfloor\frac{m}{2}\rfloor},\ast)$-codimensions is determined. As a consequence we prove that for any $G$-grading on $UT_m(F)$ and any graded involution the $(G,\ast)$-exponent is $m$ if $m$ is even and either $m$ or $m+1$ if $m$ is odd. For the algebra $UT_3(F)$ there are, up to equivalence, two non-trivial gradings that admit a graded involution: the canonical $\mathbb{Z}$-grading  and the $\mathbb{Z}_2$-grading induced by $(0,1,0)$. We determine a basis for the $(\mathbb{Z}_2,\ast)$-identities and prove that the exponent is $3$. Hence we conclude that the ordinary $\ast$-exponent for $UT_3(F)$ is $3$.
\end{abstract}
\maketitle

\section{Introduction}
Let $F$ be a field of characteristic zero and let $m$ be a positive integer. The algebra $UT_m(F)$ of $m\times m$ upper triangular matrices with entries in $F$ plays an important role in the theory of algebras with polynomial identities. A grading by a group $G$ (or a $G$-grading) on an algebra $A$ is a vector space decomposition $A=\oplus_{g\in G}A_g$ such that $A_gA_h\subseteq A_{gh}$, for every $g,h\in G$. The subspace $A_g$ is called the homogeneous component of degree $g$. If $a$ is a non-zero element of $A_g$ we say that $a$ is homogeneous of degree $g$ and denote by $\mathrm{deg}_G\, a$ the element $g$. The support of the grading is the set $\{g\in G \mid A_g\neq 0\}$. 

Let ${\bf g}=(g_1,\dots, g_m)$ be an $m$-tuple of elements of the group $G$. Then $UT_m(F)$ admits a grading for which the elementary matrix $E_{i,j}$ with $1$ in the $(i,j)$-th entry and $0$ in the remaining entries is homogeneous of degree $g_ig_j^{-1}$. We say that this is the elementary grading on $UT_m(F)$ induced by ${\bf g}$. We remark that we adopt this convention following \cite{EK}, rather then the one where $E_{i,j}$ is homogeneous of degree $g_i^{-1}g_j$. The group gradings on $UT_m(F)$ were classified by A. Valenti and M. Zaicev. 
\begin{theorem}\cite[Theorem 7]{VZ}\label{elgr}
Let $G$ be an arbitrary group and $F$ a field. Suppose that the algebra $UT_m(F)=A=\oplus_{g\in G}A_g$ of $m\times m$ upper triangular matrices over the field $F$ is $G$-graded. Then $A$, as a $G$-graded algebra, is isomorphic to $UT_m(F)$ with an elementary $G$-grading.
\end{theorem}
An involution $\ast$ on the algebra $A$ is a linear map $a\mapsto a^{\ast}$ such that $(a^{\ast})^{\ast}=a$ and $(ab)^{\ast}=b^{\ast}a^{\ast}$, for every $a,b\in A$. If $A$ has a $G$-grading $A=\oplus_{g\in G}A_g$ then $\ast$ is a graded involution if $A_g^{\ast}\subseteq A_g$, for every $g\in G$. We say that $A$ is a $G$-graded algebra with graded involution or simply that $A$ is a $(G,\ast)$-algebra. If $B$ is a $(G,\ast)$-algebra with involution $\circ$ then a homomorphism of algebras $\varphi:A\rightarrow B$ is a homomorphism of $(G,\ast)$-algebras if $\varphi(A_g)\subseteq B_g$ for every $g\in G$ and $\varphi(a^{\ast})=\varphi(a)^{\circ}$, for every $a\in A$. If $\varphi$ is an isomorphism of algebras then we say that $A$ and $B$ are isomorphic as $(G,\ast)$-algebras and that $\ast$ and $\circ$ are equivalent. The $H$-grading $A^{\prime}=\oplus_{h\in H}A_{h}^{\prime}$ on the algebra $A$ is a coarsening of the $G$-grading
\begin{align}\label{ggrading}
A=\oplus_{g\in G}A_{g}
\end{align}
if for every $g\in G$ there exists $h\in H$ such that $A_g\subseteq A_h^{\prime}$. Let $\alpha:G\rightarrow H$ be a homomorphism of groups. Then we construct an $H$-grading from the $G$-grading (\ref{ggrading}) with homogeneous components $A_h^{\prime}=\oplus_{g\in \alpha^{-1}(h)}A_g$, $h$ in $H$. We refer to this $H$-grading as the coarsening of the $G$-grading (\ref{ggrading}) induced by $\alpha$.

A basis for the (ordinary) polynomial identities for $UT_m(F)$ is well known (see, for example, \cite{GZ}). The linear transformation $\ast$ such that $E_{i,j}^\ast=E_{m+1-j,m+1-i}$ is an involution on $UT_m(F)$ called the reflection involution. If $m=2r$ let 
\begin{align}\label{D}
D=\left(\begin{array}{cc}
I_r & 0\\
0& -I_r
\end{array}\right),
\end{align}
the involution $s$ on $UT_{m}(F)$ given by $A^s=DA^{\ast}D^{-1}$, for $A\in UT_m(F)$, is called the symplectic involution. O. M. Di Vincenzo, P. Koshlukov and R. La Scala proved in \cite[Proposition 2.5]{VKS} that every involution on $UT_m(F)$ is equivalent to the reflection involution or to the symplectic involution. The polynomial identities with involution for $UT_m(F)$ were also studied in \cite{VKS}, a finite basis for the identities with involution, and other numerical invariants, were determined for $m=2,3$. The graded involutions on $UT_m(F)$, for an algebraically closed field $F$ with $\mathrm{char}\, F\neq 2$ were classified in \cite{FGDY}. The classification is given in the next result.
\begin{corollary}\cite[Corollary 34]{FGDY}\label{classf}
	Let $F$ be an algebraically closed field, $\mathrm{char}\,F\neq 2$ and let $G$ be a group. Let $\mathcal{U}^{\prime}$ be a $G$-grading on $UT_m(F)$ such that $\mathrm{supp}\ \mathcal{U}^{\prime}$ generates $G$. If $\mathcal{U}^{\prime}$ admits an antiautomorphism then $G$ is an abelian group. Let $\varphi^{\prime}$ be an involution on the graded algebra $\mathcal{U}^{\prime}$. Then $(\mathcal{U}^{\prime},\varphi^{\prime})$ is isomorphic to $(\mathcal{U},\varphi)$, where $\mathcal{U}$ is the elementary grading on $UT_m(F)$ induced by a tuple $(g_1,\dots,g_m)$ of elements of $G$ such that $g_1 + g_m=g_2 + g_{m-1}=\cdots=g_m + g_1$ and $\varphi$ is either $\ast$ or $s$.
\end{corollary}

In light of the corollary above in order to study the polynomial identities with graded involution on an algebra of upper triangular matrices we may assume without loss of generality that the grading group is abelian. Henceforth in the paper we assume that the group $G$ is abelian. Let $X_{g}=\{x_{i,g}, x_{i,g}^{\ast} \mid i\in \mathbb{N}\}$, $g\in G$, be a family of countable pairwise disjoint sets. The free algebra $F\langle X^{\prime} \rangle$, freely generated by $X^{\prime}=\cup_{g\in G}X_g$, admits an involution, also denoted by $\ast$, such that $(x_{i,g})^{\ast}=x_{i,g}^{\ast}$ and $(x_{i,g}^{\ast})^{\ast}=x_{i,g}$. The algebra $F\langle X^{\prime} \rangle$ also admits a grading by $G$ such that the indeterminates in $X_g$ are homogeneous of degree $g$. Since $G$ is an abelian group $\ast$ is a graded involution on $F\langle X^{\prime} \rangle$. Note that $F\langle X^{\prime} \rangle$ is generated, as an algebra with involution, by the set $X_G=\{x_{i,g}\mid i\in \mathbb{N}, g\in G\}$. Henceforth we denote by $F_{G,\ast}\langle X_G \rangle$ the algebra $F\langle X^{\prime} \rangle$, we may omit the group $G$ in $X_G$ and write $F_{G,\ast}\langle X \rangle$.  If $A$ is an algebra with a $G$-grading and graded involution $\ast$ then for every map $X\rightarrow A$ compatible with the $G$-grading on $A$, i.e. $x_{i,g}$ is mapped onto an element of $A_g$, there exists a unique homomorphism (of $(G,\ast)$-algebras) $\varphi:F_{G,\ast}\langle X \rangle\rightarrow A$ that extends the map $X\rightarrow A$. The polynomial $f(x_{i_1,g_1},\dots, x_{i_n,g_n})$ is a $(G,\ast)$-identity for $A$ if $f(a_1,\dots, a_n)=0$ whenever $a_1\in A_{g_1},\dots, a_n\in A_{g_n}$. Henceforth a substitution $S$ for $f(x_{i_1,g_1},\dots, x_{i_n,g_n})$ is a tuple $(a_1,\dots, a_n)\in A^{n}$ such that $a_i\in A_{g_i}$, $i=1,\dots, n$. The set of $(G,\ast)$-identities for $A$ is denoted by $T_{G\,\ast}(A)$. We remark that $T_{G\,\ast}(A)$ is a $T(G,\ast)$-ideal of $F_{G,\ast}\langle X \rangle$, i.e., it is invariant under the involution $\ast$ and the endomorphisms of $F_{G,\ast}\langle X \rangle$ as a $(G,\ast)$-algebra. Such ideals of $F_{G,\ast}\langle X \rangle$ are called $T_{G,\ast}$-ideals. If $P$ is a subset of $F_{G,\ast}\langle X \rangle$ such that the intersection of the $T(G,\ast)$-ideals of $F_{G,\ast}\langle X \rangle$ that contain $P$ is $T_{G\,\ast}(A)$ then  we say that $P$ is a basis for $T_{G\,\ast}(A)$. If $G=\{0\}$ we recover the definition of $\ast$-polynomial identity for algebras with involution. In this case we denote by $T_{\ast}(A)$ the set of $\ast$-identities for the algebra with involution $A$. 

The graded polynomial identities for $M_m(F)$ with the transpose involutions were studied in \cite{HN} for a crossed product grading and $\mathrm{char}\, F=0$ and in \cite{GC} for an elementary grading with commutative neutral component and an infinite field $F$. In \cite[Theorem 2.3]{DKV} the authors provided basis for the graded polynomial identities for $UT_{m}(F)$ with an arbitrary grading, as a consequence they proved that two gradings are isomorphic if and only if they satisfy the same graded polynomial identities. In the next proposition we prove that the analogous result holds for $(G,\ast)$-identities.
\begin{remark}\label{rem}
	The graded identities of an algebra determine the ordinary identities, see \cite[Proposition 1]{BD}. The analogous result holds for $\ast$-graded identities and the proof is similar. More precisely, if $A=\oplus_{g\in G}A_g$, $B=\oplus_{g\in G}B_g$ are $G$-graded algebras with graded involutions $\circ$, $\diamond$, respectively, such that $T_{G,\ast}(A)=T_{G,\ast}(B)$ then $T_{\ast}(A)=T_{\ast}(B)$.
\end{remark}

The subalgebra $F\langle X \rangle$ of $F_{G,\ast}\langle X \rangle$ (viewed as an algebra, without involution) generated by $X$ is the free $G$-graded algebra. If $A$ is a $G$-graded algebra with graded involution then $T_{G,\ast}(A)\cap F\langle X \rangle$ is the set of $G$-graded identities for $A$ (as a graded algebra, without involution), this set is denoted by $T_G(A)$.
\begin{remark}\label{idisom}
	Let $A$, $B$ be $G$-gradings on algebras of upper triangular matrices. Then $T_G(A)=T_G(B)$ if and only if $A$ and $B$ are isomorphic as $G$-graded algebras. Moreover if $A$ and $B$ are elementary gradings then $A=B$, see \cite{DKV},\cite{DS}.
\end{remark}

\begin{proposition}
	Let $F$ be an algebraically closed field, $\mathrm{char}\, F\neq 2$. Let $\mathcal{U}$ and $\mathcal{U}^{\prime}$ be $G$-gradings on algebras of upper triangular matrices and let $\circ$, $\diamond$ be graded involutions on $\mathcal{U}$, $\mathcal{U}^{\prime}$, respectively. Then $(\mathcal{U},\circ)$ is isomorphic to $(\mathcal{U}^{\prime},\diamond)$ if and only if $T_{G,\ast}(\mathcal{U})=T_{G,\ast}(\mathcal{U}^{\prime})$.
\end{proposition}
\begin{proof}
	We assume that $T_{G,\ast}(\mathcal{U})=T_{G,\ast}(\mathcal{U}^{\prime})$ and that $(\mathcal{U},\circ)$, $(\mathcal{U}^{\prime},\diamond)$ are as in Corollary \ref{classf} with elementary gradings. Note that $T_G(\mathcal{U})=T_G(\mathcal{U}^{\prime})$, therefore Remark \ref{idisom} implies that $\mathcal{U}=\mathcal{U}^{\prime}$. Remark \ref{rem} implies that $(\mathcal{U},\circ)$, $(\mathcal{U}^{\prime},\diamond)$ satisfy the same $\ast$-identities. Let $m$ be the positive integer such that $\mathcal{U}$ is a grading on $UT_m(F)$.  If $m$ is odd then $\circ$ and $\diamond$ coincide with the reflection involution, thus $\circ=\diamond$ and we are done. Now assume that $m$ is even. Note that in the relatively free algebra determined by $(UT_m,\circ)$ the element $[x_1,x_2]\cdots[x_{2m-3},x_{2m-2}]$ is symmetric if $\circ$ is the reflection involution and skew-symmetric if $\circ$ is the symplectic involution. Since $(UT_m,\circ)$ and $(UT_m,\diamond)$ satisfy the same identities with involution we conclude that $\circ=\diamond$.
\end{proof}

Henceforth for $\delta\in \{\emptyset, \ast\}$ the symbol $x_{i,g}^{\delta}$ represents the indeterminate $x_{i,g}$ if $\delta=\emptyset$ and it represents the indeterminate $x_{i,g}^{\ast}$ if $\delta=\ast$. We denote by $P_{n}^{G,\ast}$ the subspace
\begin{align*}
\mathrm{span}\{x_{\sigma(1),g_1}^{\delta_1}\cdots x_{\sigma(n),g_n}^{\delta_n} \mid \sigma \in S_n, g_1,\dots, g_n \in G, \delta_i=\emptyset\mbox{ or }\delta_i=\ast, i=1,\dots,n\}.
\end{align*}
The $n$-th $(G,\ast)$-codimension of the algebra $A$ with a $G$-grading and graded involution $\ast$ is the dimension $c_n^{G,\ast}(A)$ of the vector space $P_n^{G,\ast}(A):=P_n^{G,\ast}/(P_n^{G,\ast}\cap T_{G,\ast}(A))$. If the limit
\begin{align*}
\lim\limits{n\to \infty}\sqrt[n]{c_n^{(G,\ast)}(A)},
\end{align*}
exists then it is called the $(G,\ast)$-exponent of $A$ and is denoted by $\mathrm{exp}^{G,\ast}(A)$.

An adaptation of the proof of \cite[Theorem 10]{FP} yields the following result.
\begin{proposition}\label{coars}
	Let $A$ be an algebra with a $G$-grading with finite support and let $\ast$ be a graded involution on $A$. If the $H$-grading  $A^{\prime}$ is a coarsening of the $G$-grading on $A$ induced by a surjective homomorphism then
	\begin{align*}
	c_n^{\ast}(A)\leq c_n^{H,\ast}(A^{\prime})\leq c_n^{G,\ast}(A),
	\end{align*}
	where $c_n^{\ast}(A)$ is the ordinary $n$-th codimension of $A$.
\end{proposition}
\begin{proof}
Let $\psi:G\rightarrow H$ be a surjective homomorphism. For each $h\in H$ choose $g_h\in\psi^{-1}(h)$. Let $S=\{g\in G\mid A_g\neq 0\}$ be the support of the $G$-grading on $A$ and set
\begin{align*}
y_{i,h}= \left\{\begin{array}{c}
\sum_{l\in S\cap \psi^{-1}(h)}x_{i,l},\mbox{ if }S\cap \psi^{-1}(h) \neq \emptyset,\\ x_{i,g_h},\mbox{ if }S\cap \psi^{-1}(h) = \emptyset
\end{array}\right..
\end{align*}
Let $R$ be the subalgebra (with involution) of $F_{G,\ast}\langle X_G \rangle $ generated by the set $\{y_{i,g_h}\mid i\in \mathbb{N}, h\in H\}$. The algebra $R$ admits an $H$-grading such that $y_{i,h}$ is homogeneous of degree $h$ and its involution $R$ is a graded involution. The homomorphism from $F_{H,\ast}\langle X_H \rangle$ to $R$ such that $x_{i,h}\mapsto y_{i,h}$ is an isomorphism of $H$-graded algebras with involution. We identify $F_{H,\ast}\langle X_H \rangle$ with $R$, with this identification $T_{H,\ast}(A^{\prime})= T_{G,\ast}(A)\cap R$ and $P_n^{H,\ast}\subseteq P_n^{G,\ast}$. Hence the kernel of the linear map
\begin{align}\label{linmap}
P_n^{H,\ast}\hookrightarrow P_n^{G,\ast}\rightarrow P_n^{G,\ast}(A),
\end{align}
where the first map is the inclusion and the second map is the canonical quotient map, is $P_n^{H,\ast}\cap T_{G,\ast}(A)$. Note that $P_n^{H,\ast}\cap T_{G,\ast}(A)=P_n^{H,\ast}\cap T_{H,\ast}(A^{\prime})$, hence (\ref{linmap}) induces an injective linear transformation $P_n^{H,\ast}(A^{\prime})\rightarrow P_n^{G,\ast}(A)$. Clearly this implies that $c_n^{H,\ast}(A^{\prime})\leq c_n^{G,\ast}(A)$. The other inequality is a consequence of this applied to the trivial group and the homomorphism $H\rightarrow \{0\}$.
\end{proof}

The paper is organized in the following way, in Section \ref{s2} we prove that every grading on $UT_m(F)$ that admits an involution is a coarsening of a suitable $\mathbb{Z}^{\lfloor\frac{m}{2}\rfloor}$-grading for which the reflection and symplectic (if $m$ is even) involutions are graded involutions. A basis for the $(G,\ast)$-identities for this algebra with the reflection and symplectic involutions is exhibited in Theorem \ref{mainfine}. We determine the growth of the $(\mathbb{Z}^{\lfloor\frac{m}{2}\rfloor}, \ast)$-codimensions and compute the $(\mathbb{Z}^{\lfloor\frac{m}{2}\rfloor}, \ast)$-exponent for $UT_m(F)$ in Theorem \ref{expo}. As a consequence we conclude that for any $G$-grading and any graded involution on $UT_m(F)$ the $(G,\ast)$-exponent is $m$ if $m$ is even and $m$ or $m+1$ if $m$ is odd. This is the analogous for the $(G,\ast)$-identities of the results in \cite{FP} for the graded exponent of $UT_m(F)$. In the graded case the exponent with an arbitrary grading coincides with the ordinary one (for $UT_m(F)$ as an associative, Lie or Jordan algebra). In Section \ref{s3} we prove that $UT_3$ has, up to equivalence, two non-trivial gradings that admit graded involution, the canonical $\mathbb{Z}$-grading studied in Section \ref{s2} and the elementary $\mathbb{Z}_2$-grading induced by the triple $(0,1,0)$. In this case every involution is equivalent to the reflection involution. We exhibit a basis for the $(\mathbb{Z}_2,\ast)$-identities in Theorem \ref{basis}. We also prove that the $(\mathbb{Z}_2,\ast)$-exponent for $UT_3(F)$ in this case is $3$, hence Proposition \ref{coars} implies that the ordinary $\ast$-exponent for $UT_3(F)$ is $3$. We remark this is a direct consequence of the results in \cite{VKS} on the ordinary $\ast$-identities for $UT_3(F)$.

\section{Graded Identities with Involution for $UT_m(F)$}\label{s2}

In this section we construct a grading on $UT_m(F)$ that is the finest among the elementary gradings that admit graded involution. For this grading we determine a basis and compute the $\ast$-graded exponent for the reflection and the symplectic involutions. 

\begin{definition}
	Let $m$ be a positive integer. The elementary $\mathbb{Z}^{\lfloor \frac{m}{2} \rfloor}$-grading on $UT_m(F)$ induced by 
	\begin{enumerate}
	\item[] $(e_1,\dots, e_{r},0,e_{r}-e_{r-1},\dots, e_r-e_1)$ if $m=2r$,
	\item[] $(e_1,\dots, e_r,0, -e_{r},\dots, -e_{1})$ if $m=2r+1$,
\end{enumerate}
where $e_i$ is the tuple in $\mathbb{Z}^{\lfloor \frac{m}{2} \rfloor}$ with $i$-th entry equal to $1$ and remaining entries equal to $0$, is called the finest grading on $UT_m(F)$.
\end{definition}

\begin{remark}
An elementary grading by a group $G$ on $UT_m(F)$ is uniquely determined by the $m-1$-tuple 
\begin{equation}\label{tupleelem}
(\mathrm{deg}_G\, E_{1,2},\mathrm{deg}_G\, E_{2,3},\dots, \mathrm{deg}_G\, E_{m-1,m}),
\end{equation}
see \cite[Proposition 1.6]{DKV}. For $UT_m(F)$ with the finest grading by $G=\mathbb{Z}^{r}$, here $r=\lfloor \frac{m}{2} \rfloor$, the elements $f_i:=\mathrm{deg}_GE_{i,i+1}$, for $i=1,\dots, r$ are a basis of $G$ as a $\mathbb{Z}$-module and the $(m-1)$-tuple (\ref{tupleelem}) is equal to \[(f_1,f_2,\dots, f_{r-1},f_r,f_{r-1},\dots, f_2,f_1),\] if $m=2r$ and it equals \[(f_1,f_2,\dots, f_{r-1},f_r, f_r,f_{r-1},\dots, f_2,f_1),\] if $m=2r+1$.
\end{remark}

Next we prove that any grading on $UT_m(F)$ that admits a graded involution is isomorphic to a coarsening of the finest grading on this algebra. 

\begin{proposition}\label{coarsening}
	Let $A$ be a grading by a group $G$ on $UT_m(F)$. If $A$ admits a graded involution $\circ$ then there exists a homomorphism of groups $\alpha:\mathbb{Z}^{\lfloor \frac{m}{2} \rfloor}\rightarrow G$ such that $(A,\circ)$ is isomorphic to the grading induced by $\alpha$ from the finest grading on $UT_m(F)$ with the reflection or symplectic involution.
\end{proposition}
\begin{proof}
Theorem \ref{elgr} implies that there exists a tuple ${\bf g}=(g_1,\dots, g_m)$ such that $A$ is isomorphic to the elementary grading on $UT_m(F)$ induced by ${\bf g}$. Hence we may assume without loss of generality that $A$ has the elementary $G$-grading induced by ${\bf g}$. For every $g\in G$ the tuple $(g_1g,\dots, g_mg)$ induces the same the elementary grading, hence we may assume without loss of generality that $g_{r+1}=0$, where $r=\lfloor \frac{m}{2} \rfloor$. Let $\overline{F}$ be the algebraic closure of $F$. The extension of $\circ$ to $UT_m(\overline{F})$ is a graded involution on this algebra with the elementary grading induced by ${\bf g}$. Corollary \ref{classf} implies that $g_1 + g_m=g_2 + g_{m-1}=\cdots=g_m + g_1$. Hence the reflection involution $\ast$ is a graded involution on $A$. The map $a\mapsto (a^{\ast})^{\circ}$ is an automorphism of $A$ as a graded algebra. Then there exists a homogeneous invertible matrix $u\in A$ such that $(a^{\ast})^{\circ}=uau^{-1}$ (see \cite[Lemma 24]{FGDY}). Since $u$ is invertible the entries in the diagonal of $u$ are non-zero, hence $u$ is homogeneous of degree $0$. We have $a^{\circ}=ua^{\ast}u^{-1}$, since $\circ$ is an involution it follows that $u^{\ast}=\pm u$. Let $D$ be the matrix in (\ref{D}). \cite[Lemma 2.4]{VKS} implies that there exists a matrix $c\in UT_m(F)$ such that $u=cc^{\ast}$ if $u^{\ast}=u$ and $u^{\prime}=cc^{s}$, where $u^{\prime}=uD$, if $u^{\ast}=-u$. Since $u$ is homogeneous of degree $0$ it follows from the proof of \cite[Lemma 2.4]{VKS} that $c$ is homogeneous of degree $0$. The map $\psi(a)=c^{-1}a c$ is an isomorphism of $(G,\ast)$-algebras from $(A,\circ)$ to $(A,\ast)$ if $u^{\ast}=u$ and from $(A,\circ)$ to $(A,s)$ if $u^{\ast}=-u$.  

Since $g_{r+1}=0$ the elements $g_1,\dots, g_m$ lie in the support of the grading. Therefore Corollary \ref{classf} implies that the subgroup generated by $g_1,\dots, g_m$ is abelian. Let $\alpha:\mathbb{Z}^{\lfloor \frac{m}{2} \rfloor}\rightarrow G$ be the map such that $\alpha(e_i)=g_i$ for $i=1,\dots, r$. The grading induced by $\alpha$ on the finest grading on $UT_m(F)$ is the elementary grading determined by ${\bf g}$.
\end{proof}

\begin{lemma}\label{l2}
	Let $G=\mathbb{Z}^{\lfloor \frac{m}{2} \rfloor}$. If $UT_m(F)$ has the finest grading then $\mathrm{deg}_G\, E_{i,j}=\mathrm{deg}_G\, E_{k,l}\neq 0$ if and only if $E_{i,j}=E_{k,l}$ or $E_{i,j}^{\ast}=E_{k,l}$.
\end{lemma}
\begin{proof}
	Let us assume that 
	\begin{align}\label{eqdeg}
	\mathrm{deg}_G\, E_{i,j}=\mathrm{deg}_G\, E_{k,l}\neq 0.
	\end{align}
	We prove the result for $\mathcal{U}=UT_{2r}$, the proof for $UT_{2r+1}$ is analogous. We consider the following sets of elementary matrices: $E_I=\{E_{u,v}\mid u,v\leq r\}$; $E_{II}=\{E_{u,v}\mid u\leq r, v>r\}$; $E_{III}=\{E_{u,v}\mid u,v>r\}$. Note that 
	\begin{enumerate}
		\item If $E_{i,j}\in E_{I}$ then $\mathrm{deg}\, E_{i,j}=e_{i}-e_j$;
		\item If $E_{i,j}\in E_{II}$ then $\mathrm{deg}\, E_{i,j}=e_i+e_{2r+1-j}-e_{r}$;
		\item If $E_{i,j}\in E_{III}$ then $\mathrm{deg}\, E_{i,j}=e_{2r+1-j}-e_{2r+1-i}$.
	\end{enumerate}
	If $E_{i,j}\in E_{I}$ then (\ref{eqdeg}) implies that $E_{k,l}\in E_{I}\cup E_{III}$. If $E_{k,l}\in E_{I}$ then (\ref{eqdeg}) implies that $e_i-e_j=e_k-e_l$, therefore $E_{i,j}=E_{k,l}$. If $E_{k,l}\in E_{III}$ then $e_{i}-e_j=e_{2r+1-l}-e_{2r+1-k}$, hence $i=2r+1-l$ and $j=2r+1-k$. In this case $E_{k,l}=E_{i,j}^{\ast}$. Analogously if $E_{i,j}\in E_{III}$ we conclude that $E_{k,l}=E_{i,j}$ or $E_{k,l}=E_{i,j}^{\ast}$. Now assume that $E_{i,j}\in E_{II}$, then $E_{k,l}\in E_{II}$. Hence (\ref{eqdeg}) implies that $e_i+e_{2r+1-j}-e_{r}=e_{k}+e_{2r+1-l}-e_{r}$. We have two possibilities: $i=k$ and $2r+1-j=2r+1-l$ or $i=2r+1-l$ and $2r+1-j=k$. In the first case $E_{i,j}=E_{k,l}$ and in the second case $E_{k,l}=E_{i,j}^{\ast}$.
\end{proof}

\begin{remark}
For $UT_m(F)$ with the finest grading it follows from Lemma \ref{l2} that $\mathrm{dim}\, (UT_m(F))_g=1$ if and only if $g=\mathrm{deg}_G\, E_{i,j}$ where $i+j=m+1$. Then $i\leq r$, where $r=\lfloor \frac{m}{2} \rfloor$, and $g=2e_i-e_r$ if $m=2r$ and $i < r$, $g=e_r$ if $m=2r$ and $i = r$, and $g=2e_i$ if $m=2r+1$. Hence the set of elements in the support of the grading for which the corresponding homogeneous component has dimension one is $\{2e_1-e_r,\dots, 2e_{r-1}-e_r, e_r\}$ if $m=2r$ and $\{2e_1.2e_2,\dots, 2e_r\}$ if $m=2r+1$.
\end{remark}

Now assume that $UT_m(F)$ has the finest grading. The neutral component is the subspace of the diagonal matrices, hence
\begin{align}\label{i1}
[x_{1,0},x_{2,0}],
\end{align}
is a $(\mathbb{Z}^{\lfloor \frac{m}{2} \rfloor},\ast)$-identity for $UT_m(F)$. We consider the following polynomials
\begin{align}\label{i2}
x_{1,g}^{\ast}-x_{1,g}, \mbox{ where } \mathrm{dim}\, (UT_m(F))_g=1,
\end{align}
and for $m=2r+1$ the polynomials
\begin{align}\label{i3}
x_{1,g}x_{2,0}x_{3,h}-x_{1,g}x_{2,0}^{\ast}x_{3,h},\mbox{ where }g=e_i, h=e_{j}, 1\leq i,j\leq r
\end{align}
\begin{align}\label{i4}
	x_{1,0}x_{2,g}x_{3,0}-x_{1,0}^{\ast}x_{2,g}x_{3,0}-	x_{1,0}x_{2,g}x_{3,0}^{\ast}+	x_{1,0}^{\ast}x_{2,g}x_{3,0}^{\ast},\mbox{ where }g=e_i, 1\leq i \leq r.
\end{align}
and
\begin{align}\label{i5}
	x_{1,0}x_{2,g}x_{3,h}-x_{1,0}^{\ast}x_{2,g}x_{3,h},\mbox{ where }g=e_i, 2\leq i \leq r, h=e_j-e_i, 1\leq j <i.
\end{align}
\begin{proposition}\label{id}
	The polynomials (\ref{i1}), (\ref{i2}), (\ref{i3}), (\ref{i4}) and (\ref{i5}) are $(\mathbb{Z}^{\lfloor \frac{m}{2} \rfloor},\ast)$-identities for $UT_m(F)$ with the finest grading and the reflection involution.
\end{proposition}
\begin{proof}
The proof is an easy verification that the result of every elementary substitution is zero, so we omit it.
\end{proof}

\begin{lemma}\label{one}
	Let $\ast$ be a graded involution for $UT_m(F)$ with the finest grading by $G=\mathbb{Z}^{\lfloor \frac{m}{2}\rfloor}$.  Let $M=x_{u_1,g_1}^{\delta_1}\cdots x_{u_n,g_n}^{\delta_{n}}$ be a multilinear monomial such that $g_i\neq 0$ for $i=1,\dots,n$ and $n>1$. If $M\notin T_{G,\ast}(UT_m(F))$ then there exists only one elementary substitution $S$ such that $M_S\neq 0$.
\end{lemma}
\begin{proof}
	Let $S=(E_{i_1,j_1},\dots, E_{i_n,j_n})$ be an elementary substitution for $x_{u_1,g_1}\cdots x_{u_n,g_n}$. Note that $$(x_{u_1,g_1}\cdots x_{u_n,g_n})_S=M_{S^{\prime}},$$ where $S^{\prime}=(E_{i_1,j_1}^{\delta_1},\dots, E_{i_n,j_n}^{\delta_n})$. Hence we may assume without loss of generality that $M=x_{u_1,g_1}\cdots x_{u_n,g_n}$. Let $S_1,S_2$ be elementary substitutions such that $(x_1\cdots x_n)_{S_i}\neq 0$, $i=1,2$ and $S_1\neq S_2$. Let $t$ be the smallest index such that $E_{i,j}:=(x_{u_t,g_t})_{S_1}\neq (x_{u_t,g_t})_{S_2}:=E_{r,s}$. Lemma \ref{l2} implies that $E_{r,s}=E_{i,j}^{\ast}$. If $t>1$ then $M_{S_1}\neq 0$ implies that $(x_{u_{t-1},g_{t-1}})_{S_1}=E_{u,i}$ for some index $u$. Moreover since $t$ is minimal we conclude that $(x_{u_{t-1},g_{t-1}})_{S_2}=E_{u,i}$. Hence $M_{S_2}\neq 0$ implies that $E_{u,i}E_{i,j}^{\ast}\neq 0$. Therefore $m+1-j=i$, in this case $E_{r,s}=E_{i,j}^{\ast}=E_{i,j}$, which is a contradiction since $E_{i,j}\neq E_{r,s}$. Now assume that $t=1$. If $(x_{u_2,g_2})_{S_1}=(x_{u_2,g_2})_{S_2}$ then an analogous argument to the one above implies that $E_{r,s}=E_{i,j}^{\ast}=E_{i,j}$, which is a contradiction. Otherwise let $E_{j,u}=(x_{u_2,g_2})_{S_1}$. Lemma \ref{l2} implies that $(x_{u_2,g_2})_{S_2}=E_{j,u}^{\ast}$. Since $M_{S_2}\neq 0$ we conclude that 
	\begin{align*}
	E_{i,j}^{\ast}E_{j,u}^{\ast}\neq 0.
	\end{align*}
	Therefore $m+1-i=m+1-u$, which implies that $i=u$. This is a contradiction since $i<j<u$.
\end{proof}

One of the main goals in this section is to provide a finite basis for the $(G,\ast)$-identities for $UT_m(F)$ with the finest grading by $G=\mathbb{Z}^{\lfloor \frac{m}{2}}\rfloor$ and either reflection or the symplectic involution. To this end we need to study the monomial identities for $UT_m(F)$.

\begin{definition}
    Let $A$ be an algebra graded by an abelian group with a graded involution $\ast$. A monomial in $T_{G,\ast}(A)$ is called a trivial monomial identity for $A$ if it is a consequence of the identities $x_{1.g}$ where $g$ does not lie in the support of $A$.
\end{definition}

\begin{remark}
For $m>4$ on the other hand the monomial $M=x_{1,g_1}x_{2,g_2}x_{3,g_3}$, where $g_1=\mathrm{deg}\, E_{3,m-1}$, $g_2=\mathrm{deg}\, E_{m-1,m}$ and $g_3=\mathrm{deg}\, E_{m-2,m-1}$ is an identity for $UT_m(F)$. We note that $g_1, g_2, g_3, g_1+g_2, g_2+g_3, g_1+g_2+g_3$ lie in the support of the finest grading on $UT_m(F)$, therefore $M$ is not a trivial monomial identity for $UT_m(F)$.
\end{remark}

In the next results we prove that every monomial identity for $UT_m(F)$ follows from the monomial identities of degree at most $3$, moreover for $m=2,3,4$ the algebra $UT_m(F)$ has no non-trivial monomial identities.

Hencefoth in the section, $U_g$ denote the sum of the elementary matrices of degree $g$ and we assume that $U_g=0$ if $g$ does not lie in the support of the finest grading on $UT_m(F)$.

\begin{lemma}\label{MUg}
Let $M=x_{1,g_1}\cdots x_{n,g_n}$ and let $U_g$ be the sum of the elementary matrices of degree $g$. Let $S$ be the substitution such that $x_{t,g_t}$ is replaced with $U_{g_t}$ for $t=1,\dots, n$. Then $M\in T_{G,\ast}(UT_m(F))$ if and only if $M_S=0$.
\end{lemma}

\begin{proof}
Note that $S$ is an admissible substitution for $M$, hence $M_S=0$ if $M\in T_{G,\ast}(UT_m(F))$. Next we prove the converse.
Let $S_1,\dots, S_k$ be the all elementary substitutions for $M$. We have \[0=M_S=U_{g_1}\cdots U_{g_t}=\sum_{l=1}^kM_{S_l}=\sum_{1\leq i\leq j\leq m}n_{i,j}E_{i,j},\] where $n_{i,j}$ is the number of indices $l$ such that $M_{S_l}=E_{i,j}$. The field $F$ has characteristic zero, hence we conclude that $n_{i,j}=0$ for all $i,j$, therefore $M_{S_l}=0$ for $l=1,\dots, k$. Since $M$ is a multilinear monomial this implies that $M\in T_{G, \ast}(UT_m(F))$. 
\end{proof}

\begin{corollary}\label{Mlinear}
Let $M=x_{i_1,g_1}^{\delta_1}\cdots x_{i_n,g_n}^{\delta_n}$ and let $\tilde{M}=x_{1,h_1}\cdots x_{n^{\prime},h_{n^{\prime}}}$, where $(h_1,\dots, h_{n^{\prime}})$ is the tuple obtained from $(g_1,\dots, g_n)$ by deleting the entries equal to $0\in G$. Then $M\in T_{G,\ast}(UT_m(F))$ if and only if $\tilde{M}=T_{G,\ast}(UT_m(F))$.
\end{corollary}

\begin{proof}
If $g_i = 0$, all $i = 1, \ldots, n$, then $\tilde{M} = 1$. Hence, the result is valid. Hence we way assume without loss of generality that there is $g_i \neq 0$, for $i = 1, \ldots, n$.

Note that $M$ is a consequence of $\tilde{M}$, therefore $M\in T_{G,\ast}(UT_m(F))$ if $\tilde{M}\in T_{G,\ast}(UT_m(F))$. We assume now that $M\in T_{G,\ast}(UT_m(F))$. Let $S$ be the substitution such that any indeterminate of degree $g$ is replaced by $U_g$ for all $g$ in the grading group. Since $U_0$ is the identity matrix we have $0=M_S=U_{g_1}\cdots U_{g_n}=U_{h_1}\cdots U_{h_{n^{\prime}}} = \Tilde{M}_S$, hence Lemma \ref{MUg} implies that $\tilde{M}\in T_{G,\ast}(UT_m(F))$.
\end{proof}

\begin{lemma}\label{n3}
Let $M=x_{1,g_1}\cdots x_{n,g_n}$. If $M\in T_{G,\ast}(UT_m(F))$ and $n\geq 4$ then at least one of the monomials $x_{2,g_2}\cdots x_{n,g_n}$, $x_{1,g_1}\cdots x_{n-1,g_{n-1}}$ lie in $T_{G,\ast}(UT_m(F))$.
\end{lemma}

\begin{proof}
Assume that the monomials $N=x_{2,g_2}\cdots x_{n,g_n}$, $N^{\prime}=x_{1,g_1}\cdots x_{n-1,g_{n-1}}$ are not $(G,\ast)$-identities for $UT_m(F)$. Then there exist elementary substitutions $S=(a_1,\dots, a_{n-1})$ and $S^{\prime}=(a_2^{\prime},\dots, a_t^{\prime})$ for $N$ and $N^{\prime}$, respectively, such that $N_S=a_1\cdots a_{t-1}\neq 0$ and $(N^{\prime})_{S^{\prime}}=a_2^{\prime}\cdots a_{t}^{\prime}\neq 0$. Then $(a_2,\dots, a_{t-1})$ and $(a_2^{\prime},\dots, a_{t-1}^{\prime})$ are elementary substitutions for $x_{2,g_2}\cdots x_{t-1,g_{t-1}}$ that do not result in $0$. This monomial has degree $\geq 2$, hence Lemma \ref{one} implies that $(a_2,\dots, a_{t-1})=(a_2^{\prime},\dots, a_{t-1}^{\prime})$. Then $(a_1,\dots, a_{t-1},a_t^{\prime})$ is a substitution for $x_{1,g_1}\cdots x_{n,g_n}$ that does not result in zero, a contradiction since this monomial lies in $T_{G,\ast}(UT_m(F))$.
\end{proof}

\begin{proposition}\label{p3}
The monomial identities for $UT_m(F)$ with the finest grading are consequence of the monomial identities of the form $x_{1,g_1}\cdots x_{n,g_n}$ where $g_i\neq 0$ for all $i$ and $n\leq 3$.
\end{proposition}

\begin{proof}
Let $M$ be a monomial identity for $UT_m(F)$. Corollary \ref{Mlinear} implies that there exists a monomial $\tilde{M}=x_{1,h_1}\cdots x_{n^{\prime}, h_{n^{\prime}}}$ in $T_{G,\ast}(UT_m(F))$ such that $M$ is a consequence of $\tilde{M}$. Lemma \ref{n3} implies that there exists $N=x_{1,g_1}\cdots x_{n,g_n}\in T_{G,\ast}$ such that $n\leq 3$ and $\tilde{M}$ is a consequence of $N$. Hence $M$ is a consequence of $N$.
\end{proof} 

\begin{corollary}\label{nonontriv}
The algebra $UT_m(F)$ with the finest grading satisfies no non-trivial monomial identities if $m\leq 4$.
\end{corollary}

\begin{proof}
The result follows from Proposition \ref{p3} if we prove that every identity for $UT_m(F)$ of the form $x_{1,g_1}\cdots x_{n,g_n}$, where $g_i\neq 0$ for all $i$ and $n\leq 3$, is trivial. For $m=2,3$ one checks this directly. We assume now that $m=4$, then the support of the finest grading on $UT_4(F)$ is $S=\{0,e_1,e_2,e_1-e_2, 2e_1-e_2\}$. One can verify the following claim directly:

{\bf Claim:} If $g_1, g_2, g_1+g_2\in S$ then $g_1\neq g_2$ and $\{g_1,g_2\}=\{e_1,e_1-e_2\}$ or $\{g_1,g_2\}=\{e_2, e_1-e_2\}$.

In either case $x_{g_1}x_{g_2}\notin T_{G,\ast}(UT_m(F))$, hence every identity for $UT_m(F)$ of the form $x_{g_1}x_{g_2}$ is trivial. 

Now assume that $g_1,g_2,g_3\in S$ are such that $g_1+g_2, g_1+g_2+g_3\in S$. The claim above implies that $g_1+g_2\in \{2e_1-e_2, e_1\}$, since $(g_1+g_2)+g_3\in S$ we conclude that $g_1+g_2=e_1$, and therefore $\{g_1,g_2\}=\{e_2, e_1-e_2\}$ and $g_3=e_1-e_2$, then it follows that $x_{1,g_1}x_{2,g_2}x_{3,g_3}\notin T_{G,\ast}(UT_m(F))$. Hence every identity for $UT_m(F)$ of the form $x_{g_1}x_{g_2}$ is trivial.
\end{proof}

\begin{definition}\label{defgood}
Let $G=\mathbb{Z}^{\lfloor \frac{m}{2} \rfloor}$ and let $\mathcal{U}$ be the algebra $UT_m(F)$ with the finest grading. Let
\begin{align}\label{good}
M=M_1x_{u_1,g_1}^{\delta_1}M_2x_{u_2,g_2}^{\delta_2}\cdots M_kx_{u_k,g_k}^{\delta_k}M_{k+1},
\end{align}
be a monomial in $P_n^{G,\ast}$, where $g_1,\dots,g_k\neq 0$ and
\begin{align*}
M_i=x_{k_{i,1},0}^{\delta_{i,1}}\cdots x_{k_{i,s_i},0}^{\delta_{i,s_i}}.
\end{align*}
We say that $M$ is good if it is not a $(G,\ast)$-identity for $\mathcal{U}$ and moreover,
\begin{enumerate}
	\item[I] $k_{i,1}<\cdots < k_{i,s_i}$;
        \vspace{0,2cm}
	\item[II] If $\mathrm{dim}\, \mathcal{U}_{g_i}=1$ then $\delta_i=\emptyset$ and $M_i=1$;
         \vspace{0,2cm}
	\item[III]  If $\mathrm{dim}\, \mathcal{U}_g=1$ for $g=g_i + g_{i+1} + \cdots + g_j$ for $i< j$ then $M_i=1$ and $u_i<u_j$;
         \vspace{0,2cm}
	\item[IV] If $m=2r+1$, $g_i, g_{i+1}\in\{e_1,\dots, e_r\}$ then ${\delta_{i+1,1}}=\dots={\delta_{i+1,s_{i+1}}}=\emptyset$;
         \vspace{0,2cm}
	\item[V] If $m=2r+1$, $k\geq 2$, $g_1 \in \{e_1,\dots,e_r\}, g_{2}\notin\{e_1,\dots,e_r\}$ then ${\delta_{1,1}}=\dots={\delta_{1,s_{1}}}=\emptyset$;
         \vspace{0,2cm}
	\item[VI] If $m=2r+1$, $k\geq 2$, $g_{k-1}\notin\{e_1,\dots,e_r\}, g_{k}\in\{e_1,\dots,e_r\}$ then ${\delta_{k+1,1}}=\dots={\delta_{k+1,s_{k+1}}}=\emptyset$;
         \vspace{0,2cm}
	\item[VII] If $m=2r+1$, $k=1$ and $g_1\in\{e_1,\dots,e_r\}$ then $\delta_{1,1}=\dots = \delta_{1,s_1}=\emptyset$ or 	$\delta_{2,1}=\dots = \delta_{2,s_2}=\emptyset$.	
\end{enumerate}
\end{definition}

\begin{lemma}\label{goodsub}
	Let $G=\mathbb{Z}^{\lfloor \frac{m}{2} \rfloor}$ and let $\mathcal{U}$ be the algebra $UT_m(F)$ with the finest grading. Let $M$ be a good monomial written as in (\ref{good}) with $k>1$ indeterminates of degree different from $0$. Let $M^{\prime}$ be a good monomial in the same indeterminates as $M$. If there exists an elementary substitution $S$ by elements of $\mathcal{U}$ such that $(M)_S=(M^{\prime})_S\neq 0$ then $M=M^{\prime}$.
\end{lemma}
\begin{proof}
	Since $M^{\prime}$ is a monomial in the same indeterminates as $M$ there exists a permutation $\sigma \in S_k$ such that 
	\begin{align*}
		M^{\prime}=M_1^{\prime}x_{u_{\sigma(1)},g_{\sigma(1)}}^{\epsilon_{\sigma(1)}}M_2^{\prime}x_{u_{\sigma(2)},g_{\sigma(2)}}^{\epsilon_{\sigma(2)}}\cdots M_k^{\prime}x_{u_{\sigma(k)},g_{\sigma(k)}}^{\epsilon_{\sigma(k)}}M_{k+1}^{\prime},
	\end{align*}
	where 
	\begin{align*}
		M_i^{\prime}=x_{l_{i,1},0}^{\epsilon_{i,1}}\cdots x_{l_{i,t_i},0}^{\epsilon_{i,t_i}}.
	\end{align*}

The proof will be divided into several claims below that yield the result.

\vspace{0,5cm}
		{\bf Claim A:} \textit{$\sigma(k)=k$.}
		\vspace{0,3cm}
		
		We prove Claim A by contradiction. Let $a=\sigma(k)$, $b=\sigma^{-1}(k)$ and assume that $a<k$, then $b<k$. Now let $t,u,v,w$ be the indices such that $(x_{u_a,g_a}^{\delta_a})_S=E_{t,u}$ and $(x_{u_k,g_k}^{\delta_k})_S=E_{v,w}$. Since $(M)_S\neq 0$ and $a<k$ we conclude that $t<u\leq v<w$ and that $(M)_S=E_{r,w}$ for some index $r$. 
		
		\vspace{0,5cm}
		{\bf Claim A.1:} \textit{$\epsilon_a\neq \delta_a$}
		\vspace{0,3cm}
		
		Indeed if $\epsilon_a=\delta_a$ then $(x_{u_a,g_a}^{\epsilon_a})_S=E_{t,u}$, hence since $(M^{\prime})_S\neq 0$ and $\sigma(k)=a$ we have $(M^{\prime})_S=E_{r^{\prime},u}$ for some index $r^{\prime}$, then we have $$E_{r,w}=(M^{\prime})_S=(M)_S=E_{r^{\prime},u}.$$ Then $u=w$, this is a contradiction since $u<w$. Then $\epsilon_a\neq \delta_a$.
		
		\vspace{0,3cm} 
		As a consequence of Claim A.1, $$(x_{u_a,g_a}^{\epsilon_a})_S=(x_{u_a,g_a}^{\delta_a})_S^{\ast}=E_{t,u}^{\ast}=E_{m+1-u,m+1-t}.$$ Once again $(M^{\prime})_S\neq 0$ and $\sigma(k)=a$ imply that $(M^{\prime})_S=E_{r^{\prime},m+1-t}$ for an appropriate index $r^{\prime}$. Then we have $$E_{r,w}=(M)_S=(M^{\prime})_S=E_{r^{\prime},m+1-t}$$ and consequently $w=m+1-t$. This last equality implies the next statement
		
		\vspace{0,5cm}
		{\bf Claim A.2:} \textit{$\mathrm{dim}\, \mathcal{U}_g=1$ for $g=g_a + g_{a+1} + \cdots + g_k$.}
		\vspace{0,3cm}
		
		Indeed we recall that $(x_{u_a,g_a}^{\delta_a})_S=E_{t,u}$ and $(x_{u_k,g_k}^{\delta_k})_S=E_{v,w}$, hence since $(M)_S\neq0$ we have
		\begin{align*}
			0\neq\left(x_{u_a,g_a}^{\delta_a}M_{a+1}x_{u_{a+1},g_{a+1}}^{\delta_{a+1}}\cdots x_{u_k,g_k}^{\delta_k}\right)_S=E_{t,u}(M_{a+1})_S(x_{u_{a+1},g_{a+1}}^{\delta_{a+1}})_S\cdots E_{v,w}=E_{t,w}
		\end{align*}
		This implies that $E_{t,w}$ is homogeneous of degree $g=g_a + g_{a+1} + \cdots + g_k$. Recall that $t+w=m+1$, hence $E_{t,w}^{\ast}=E_{t,w}$. Thus Lemma \ref{l2} implies that $\mathrm{dim}\, \mathcal{U}_g=1$. 
		
		\vspace{0,5cm}
		{\bf Claim A.3:} \textit{$\delta_k\neq\epsilon_k$}
		\vspace{0,3cm}
		
		The proof is analogous to the proof of Claim A.1. Recall that $b=\sigma^{-1}(k)<k$ and note that since $(M^{\prime})_S\neq 0$ we have
		\begin{align*}
			0\neq \left(x_{u_{k},g_{k}}^{\epsilon_{k}}M_{b+1}^{\prime}\cdots M_k^{\prime}x_{u_{a},g_{a}}^{\epsilon_{a}}\right)_S=\left(x_{u_{k},g_{k}}^{\epsilon_{k}}\right)_S(M_{b+1}^{\prime})_S\cdots \left(x_{u_{a},g_{a}}^{\epsilon_{a}}\right)_S=\left(x_{u_{k},g_{k}}^{\epsilon_{k}}\right)_S\cdots E_{m+1-u,m+1-t}
		\end{align*}
		If $\delta_k=\epsilon_k$ then $\left(x_{u_{k},g_{k}}^{\epsilon_{k}}\right)_S=E_{v,w}$ and the above reasoning implies that $v<w\leq m+1-u<m+1-t$ which is a contradiction since $w=m+1-t$. Hence $\delta_k\neq\epsilon_k$ and the substitution $S$ in the monomial $x_{u_{k},g_{k}}^{\epsilon_{k}}M_{b+1}^{\prime}\cdots M_k^{\prime}x_{u_{a},g_{a}}^{\epsilon_{a}}$ yields $E_{m+1-w,m+1-t}$.
		
		\vspace{0,5cm}
		{\bf Claim A.4:} \textit{$\mathrm{dim}\, \mathcal{U}_h=1$ for $h=g_{\sigma(b)} + g_{\sigma(b+1)} + \cdots + g_{\sigma(k)}$.}
		\vspace{0,3cm}
		
		Indeed, since the matrix  $E_{m+1-w,m+1-t}$ is homogeneous of degree $$h=g_{k} + g_{\sigma(b+1)} + \cdots + g_{a}=g_{\sigma(b)} + g_{\sigma(b+1)} + \cdots + g_{\sigma(k)},$$ and since $t+w=m+1$ we conclude that $\mathrm{dim}\, \mathcal{U}_h=1$.  
		
		\vspace{0,3cm}
		Claim A.2 together with Condition (III) in Definition \ref{defgood} implies that $u_a<u_k$. Analogously Claim A.4 implies that $u_k<u_a$, thus we reach a contradiction. The contradiction arises from our supposition that $a=\sigma(k)<k$. Hence we conclude that $\sigma(k)=k$. 
		
		\vspace{0,5cm}
		{\bf Claim B:} \textit{$\sigma$ is the identity permutation.}
		\vspace{0,3cm}
		
		We prove Claim B by induction on $k$, for $k=2$ this follows directly from Claim A. For $k>2$ we note that the equality $\sigma(k)=k$ and the equality $(M)_S=(M^{\prime})_S\neq 0$ imply that 
		\begin{align*}
			0\neq\left(M_1x_{u_1,g_1}^{\delta_1}M_2x_{u_2,g_2}^{\delta_2}\cdots x_{u_{k-1},g_{k-1}}^{\delta_{k-1}}M_k\right)_S=\left(M_1^{\prime}x_{u_{\sigma(1)},g_{\sigma(1)}}^{\epsilon_{\sigma(1)}}M_2^{\prime}x_{u_{\sigma(2)},g_{\sigma(2)}}^{\epsilon_{\sigma(2)}}\cdots x_{u_{\sigma(k-1)},g_{\sigma(k-1)}}^{\delta_{\sigma(k-1)}}M_k^{\prime}\right)_S,
		\end{align*}
		now the induction hypothesis implies that $\sigma$ is the identity permutation.

\vspace{0,5cm}
{\bf Claim C:} \textit{There exists indices $i_1<i_2<\cdots<i_k<i_{k+1}$ such that $$\left(x_{u_{l},g_{l}}^{\delta_{l}}\right)_S=E_{i_l,i_{l+1}}=\left(x_{u_{l},g_{l}}^{\epsilon_{l}}\right)_S$$ for $l=1,\dots, k$.}
 \vspace{0,3cm}

Since $(M)_S\neq 0$ there exist indices $i_1<i_2<\cdots<i_k<i_{k+1}$ such that $\left(x_{u_{l},g_{l}}^{\delta_{l}}\right)_S=E_{i_l,i_{l+1}}$ for $l=1,\dots,k$. Analogously $(M^{\prime})_S\neq 0$ implies that there exist $j_1<j_2<\cdots<j_k<j_{k+1}$ such that $\left(x_{u_{l},g_{l}}^{\epsilon_{l}}\right)_S=E_{j_l,j_{l+1}}$. Note that the finest grading $\mathcal{U}$ is induced by a tuple of pairwise distinct elements of the group, therefore given an index $i$ and an element $g$ of the group $G$ there exists at most one index $j$ such that $E_{i,j}$ is homogeneous of degree $g$. Since $E_{i_l,i_{l+1}}$ is homogeneous of degree $g_l$ we conclude that the sequence $i_1<i_2<\cdots<i_{k+1}$ is determined by $i_1$ and the $k$-tuple $(g_1,\dots, g_k)$. Analogously the sequence $j_1<\cdots<j_{k+1}$ is determined by $j_1$ and $(g_1,\dots, g_k)$. Since $E_{i_1,i_{k+1}}=(M)_S=(M^{\prime})_S=E_{j_1,j_{k+1}}$ we conclude that $i_1=j_1$, therefore we have $i_l=j_l$ for $l=1,\dots,k+1$. As a consequence we conclude that $\left(x_{u_{l},g_{l}}^{\delta_{l}}\right)_S=E_{i_l,i_{l+1}}=\left(x_{u_{l},g_{l}}^{\epsilon_{l}}\right)_S$ for $l=1,\dots, k$. 

\vspace{0,5cm}
{\bf Claim D:} \textit{$\delta_i=\epsilon_i$ for $i=1,\dots, k$.}
 \vspace{0,3cm}

We prove Claim D by contradiction. Assume that $\delta_l\neq\epsilon_l$ for some $l$, then $x_{u_{l},g_{l}}^{\epsilon_{l}}=(x_{u_{l},g_{l}}^{\delta_{l}})^{\ast}$ therefore $$E_{i_l,i_{l+1}}=\left(x_{u_{l},g_{l}}^{\delta_{l}}\right)_S=(x_{u_{l},g_{l}}^{\delta_{l}})_S^{\ast}=E_{i_l,i_{l+1}}^{\ast}.$$ Since $E_{i_l,i_{l+1}}$ is homogeneous of degree $g_l\neq 0$ and $E_{i_l,i_{l+1}}^{\ast}=E_{i_l,i_{l+1}}$ we conclude from Lemma \ref{l2} that $\mathrm{dim}\, \mathcal{U}_{g_l}=1$. Then since $M$ and $M^{\prime}$ are good monomials Condition II in Definition \ref{defgood} implies that $\delta_l=\emptyset=\epsilon_l$, this is a contradiction since we assumed that $\delta_l\neq\epsilon_l$. Hence $\epsilon_l=\delta_l$ for $l=1,\dots, k$.

\vspace{0,5cm}

Claims B and D imply that we may write $M^{\prime}$ as 	

\begin{align*}	M^{\prime}=M_1^{\prime}x_{u_1,g_1}^{\delta_1}M_2^{\prime}x_{u_2,g_2}^{\delta_2}\cdots M_k^{\prime}x_{u_k,g_k}^{\delta_k}M_{k+1}^{\prime}.
\end{align*}

\vspace{0,5cm}
{\bf Claim E:} We have $$(x_{k_{a,j},0}^{\delta_{a,j}})_S=E_{i_a,i_a}=(x_{l_{a,j},0}^{\epsilon_{a,j}})_S$$ for $1\leq a \leq k+1$, $1\leq j \leq s_a$.
 \vspace{0,3cm}

Since $(M^{\prime})_S\neq 0$ and $(x_{u_{a},g_{a}}^{\delta_{a}})_S=E_{i_a,i_{a+1}}$ we conclude that the result of the substitution $S$ in the factor $x_{l_{a,j},0}^{\epsilon_{a,j}}$ of $M_a^{\prime}$ is $E_{i_a,i_a}$ for $j=1,\dots, t_a$, i. e., $(x_{l_{a,j},0}^{\epsilon_{a,j}})_S=E_{i_a,i_a}$. We recall that $M$ is written as in (\ref{good}), hence we also have $(x_{k_{a,j},0}^{\delta_{a,j}})_S=E_{i_a,i_a}$ for $j=1,\dots, s_a$. 

\vspace{0,5cm}
{\bf Claim F:} \textit{If $M_a\neq 1$ and $x_{k_{a,j},0}$ is an indeterminate in $M_a$ then $x_{k_{a,j},0}$ is an indeterminate in $M_a^{\prime}$.}
 \vspace{0,3cm}

Let $a$ be an index such that $M_a\neq 1$ and let $x_{k_{a,j},0}$ be an indeterminate in $M_a$. Recall that $M$ and $M^{\prime}$ are monomials in the same indeterminates, therefore there exists an index $b$ such that $x_{k_{a,j},0}$ is an indeterminate in $M_b^{\prime}$, in this case $k_{a,j}=l_{b,j^{\prime}}$ for some $j^{\prime}\leq t_b$. Assume that $b\neq a$, then $\delta_{a,j}\neq \epsilon_{b,j^{\prime}}$. Indeed if $\delta_{a,j}=\epsilon_{b,j^{\prime}}$ then $x_{k_{a,j},0}^{\delta_{a,j}}=x_{l_{b,j^{\prime}},0}^{\epsilon_{b,j^{\prime}}}$. Then we have $$E_{i_a,i_a}=(x_{k_{a,j},0}^{\delta_{a,j}})_S=(x_{l_{b,j^{\prime}},0}^{\epsilon_{b,j^{\prime}}})_S=E_{i_b,i_b}.$$ As a consequence we have $i_a=i_b$, this is a contradiction since $a\neq b$. Then $\delta_{a,j}\neq \epsilon_{b,j^{\prime}}$, therefore $x_{k_{a,j},0}^{\delta_{a,j}}=(x_{l_{b,j^{\prime}},0}^{\epsilon_{b,j^{\prime}}})^{\ast}$. This last equality implies that $$E_{i_a,i_a}=(x_{k_{a,j},0}^{\delta_{a,j}})_S=(x_{l_{b,j^{\prime}},0}^{\epsilon_{b,j^{\prime}}})_S^{\ast}=E_{i_b,i_b}^{\ast}=E_{m+1-i_b,m+1-i_b},$$ hence $i_a=m+1-i_b$. If $a<b$ then $i_a<i_b$ and $E_{i_a,i_b}=E_{i_a,i_{a+1}}\cdots E_{i_{b-1},i_b}$ is homogeneous of degree $g=g_a + \cdots + g_{b-1}$. The equality $i_a=m+1-i_b$ implies that $E_{i_a,i_b}^{\ast}=E_{i_a,i_b}$, hence $\mathrm{dim}\, \mathcal{U}_g=1$. Thus it follows from Condition II (if $b=a+1$) or Condition III (if $b>a+1$) in Definition \ref{defgood} that $M_a=1$, this contradicts the fact that $M_a\neq 1$. If $b<a$ we conclude analogously that $M_b^{\prime}=1$ which is a contradiction. Hence we must have $b=a$.

\vspace{0,5cm}
{\bf Claim G:} \textit{The equalitites
\begin{align*}
	M_i^{\prime}=x_{k_{i,1},0}^{\epsilon_{i,1}}\cdots x_{k_{i,s_i},0}^{\epsilon_{i,s_i}},
\end{align*}
hold for $i=1,\dots, k+1$. }
 \vspace{0,3cm}

We repeat the proof of Claim G with the obvious modifications to prove that if $M_a^{\prime}\neq 1$ for some index $a$ then every indeterminate in $M_a^{\prime}$ is an indeterminate in $M_a$. Hence we conclude that $M_a$ and $M_a^{\prime}$ are monomials in the same indeterminates for $a=1,\dots, k$. This together with Condition I in Definition \ref{defgood} imply Claim G. 

\vspace{0,5cm}
{\bf Claim H:} \textit{$\delta_{a,j}=\epsilon_{a,j}$ for $1\leq a \leq k+1$, $1\leq j \leq s_a$.}
 \vspace{0,3cm}

Note that $(x_{k_{a,j},0}^{\delta_{a,j}})_S=E_{i_a,i_a}=(x_{k_{a,j},0}^{\epsilon_{a,j}})_S$. Assume that $\delta_{a,j}\neq\epsilon_{a,j}$ for some $a$ and some $j$. Then we have $x_{k_{a,j},0}^{\delta_{a,j}}=(x_{k_{a,j},0}^{\epsilon_{a,j}})^{\ast}$, hence $$E_{i_a,i_a}=(x_{k_{a,j},0}^{\delta_{a,j}})_S=(x_{k_{a,j},0}^{\epsilon_{a,j}})_S^{\ast}=E_{i_a,i_a}^{\ast}=E_{m+1-i_a,m+1-i_a}.$$ Therefore we conclude that $2i_a=m+1$. If $m=2r$ this last equality already yields a contradiction, Claim H holds in this case. Now assume that $m=2r+1$, then $i_a=r+1$. Let $e_i\in G$ be the tuple with $i$-th entry equal to $1$ and remaining entries equal to $0$. Note that the $G$-degree of $E_{i,j}$ lies in $\{e_1,\dots, e_r\}$ if and only if $i=r+1$ or $j=r+1$. Now we consider three possible cases for $a$: $a=1$, $a=k+1$ and $1<a<k+1$. Recall that $k>1$, hence if $a=1$ then $g_2$ is the degree of $E_{i_2,i_3}$, since $r+1=i_1<i_2<i_3$ we conclude that $g_2\notin \{e_1,\dots, e_r\}$. Condition V in Definition \ref{defgood} implies that $\delta_{1,j}=\epsilon_{1,j}=\emptyset$, this contradicts our hypothesis that $\delta_{a,j}\neq\epsilon_{a,j}$. Analogously if $a=k+1$ then $g_{k-1}$ is the degree of $E_{i_{k-1},i_k}$, since $i_{k-1}<i_k<i_{k+1}=r+1$ we conclude that $g_{k-1}\notin\{e_1,\dots, e_r\}$. Now Condition VI in Definition \ref{defgood} implies that $\delta_{k+1,j}=\epsilon_{k+1,j}=\emptyset$, this is a contradiction. Now assume that $1<a<k+1$. Then $g_{a-1}$ is the degree of $E_{i_{a-1},i_a}$ and $g_{a}$ is the degree of $E_{i_a,i_{a+1}}$. Since $i_a=r+1$ we have $g_{a-1},g_a\in \{e_1,\dots, e_r\}$, then Condition IV in Definition \ref{defgood} implies that $\delta_{a,j}=\epsilon_{a,j}=\emptyset$ and this is a contradiction.

\vspace{0,5cm}

Clearly the claims above imply that $M=M^{\prime}$.

\end{proof}

\begin{lemma}\label{genpn}
	Let $G=\mathbb{Z}^{\lfloor \frac{m}{2} \rfloor}$. The set of good monomials of degree $n$ generates the vector space $P_n^{G,\ast}$ modulo the $T_{G,\ast}$-ideal generated by the identities (\ref{i1}), (\ref{i2}), (\ref{i3}), (\ref{i4}) and (\ref{i5}) and the monomials in $T_{G,\ast}(UT_m(F))$.
\end{lemma}
\begin{proof}
	Let $J$ be the $T_{G,\ast}$-ideal generated by the identities (\ref{i1})-(\ref{i5}) together with the monomials in $T_{G,\ast}(UT_m(F))$ and let $M$ be a monomial in $P_n^{G,\ast}\setminus T_{G,\ast}(UT_m)$. The monomial $M$ in $P_n^{G,\ast}$ may be written as in (\ref{good}). Now we prove that $M$ is congruent modulo $J$ to a monomial that satisfies Conditions (I)-(VII) in Definition \ref{defgood}. 
 
 To assist the reader, we will add one condition at a time.

 \vspace{0,5cm}
		{\bf Condition (I)}
		\vspace{0,3cm}
 
 For condition (I), just use (\ref{i1}) to write $M$ modulo $J$ as a monomial that satisfies this condition.

  \vspace{0,5cm}
		{\bf Condition (II)}
		\vspace{0,3cm}

If $\mathrm{dim}\, \mathcal{U}_{g_i}=1$ then $x_{u_i,g_i}^{\ast}- x_{u_i,g_i}$ is an identity for $UT_m(F)$ then we replace $x_{u_i,g_i}^{\ast}$ by $x_{u_i,g_i}$, if necessary, and assume without loss of generality that $\delta_i=\emptyset$. Note that $$M_ix_{u_i,g_i}\equiv_Jx_{u_i,g_i}^{\ast}M_i^{\ast}\equiv_Jx_{u_i,g_i}M_i^{\ast},$$ hence we may obtain a monomial that is congruent modulo $J$ to $M$, and that satisfies (I) and (II). 

  \vspace{0,5cm}
		{\bf Condition (III)}
		\vspace{0,3cm}

Now assume that $\mathrm{dim}\, \mathcal{U}_g=1$ for $g=g_i + g_{i+1} + \cdots + g_j$, where $i< j$, in this case $x_{1,g}^{\ast}-x_{1,g}\in T_{G,\ast}(UT_m(F))$. Then
	\begin{align*}
	M_ix_{u_i,g_i}^{\delta_i}\cdots M_jx_{u_j,g_j}^{\delta_j}\equiv_J(x_{u_j,g_j}^{\delta_j})^{\ast}M_j^{\ast}\cdots (x_{u_i,g_i}^{\delta_i})^{\ast}M_i^{\ast}\equiv_Jx_{u_i,g_i}^{\delta_i}\cdots M_jx_{u_j,g_j}^{\delta_j}M_{i}^{\ast},
	\end{align*}
	therefore $M$ is congruent modulo $J$ to a monomial that satisfies $(I)-(III)$.
 
   \vspace{0,5cm}
		{\bf Condition (IV)}
		\vspace{0,3cm}

 If $m=2r$ then this monomial is already good. Now assume that $m=2r+1$. If $g_i, g_{i+1}\in\{e_1,\dots, e_r\}$ then the polynomial
	\begin{align*}
	x_{1,g_i}x_{2,0}x_{3,g_{i+1}}-x_{1,g_i}x_{2,0}^{\ast}x_{3,g_{i+1}}
	\end{align*}
	is a $(G,\ast)$-identity for $UT_m(F)$ in (\ref{i3}). Hence if $\delta_{i,l}=\ast$ we have
	\begin{align*}
	x_{u_i,g_i}^{\delta_i}x_{k_{i,1},0}^{\delta_{i,1}}\cdots x_{k_{i,l},0}^{\ast}\cdots x_{k_{i,s_i},0}^{\delta_{i,s_i}} x_{u_{i+1},g_{i+1}}^{\delta_{i+1}}\equiv_J x_{u_i,g_i}^{\delta_i}x_{k_{i,1},0}^{\delta_{i,1}}\cdots x_{k_{i,l},0}\cdots x_{k_{i,s_i},0}^{\delta_{i,s_i}} x_{u_{i+1},g_{i+1}}^{\delta_{i+1}},
	\end{align*}
	therefore $x_{u_i,g_i}^{\delta_i}M_i x_{u_{i+1},g_{i+1}}^{\delta_{i+1}}\equiv_J x_{u_i,g_i}^{\delta_i}x_{k_{i,1},0}\cdots x_{k_{i,s_i},0} x_{u_{i+1},g_{i+1}}^{\delta_{i+1}}$. In this way we conclude that $M$ is congruent modulo $J$ to a monomial that satisfies $(I)-(IV)$.

   \vspace{0,5cm}
		{\bf Condition (V)}
		\vspace{0,3cm}

 Now assume that the hypothesis in $(V)$ holds, i.e., $k\geq 2$, $g_1\in\{e_1,\dots,e_r\}$, $g_2 \notin \{e_1,\dots,e_r\}$. Since $M\notin T_{G,\ast}(UT_m(F))$ there exists an elementary substitution $S$ such that $(M)_S\neq 0$. Let $a<b<c$ be the indices such that $(x_{u_1,g_1}^{\delta_1})_S=E_{a,b}$ and $(x_{u_2,g_2}^{\delta_2})_S=E_{b,c}$. Then $g_1$ is the $G$-degree of $E_{a,b}$ and therefore $a=r+1$ or $b=r+1$. Since $E_{b,c}$ is homogeneous of degree $g_2\notin \{e_1,\dots, e_r\}$ we conclude that $b\neq r+1$, hence $a=r+1$. The tuple that induced the elementary grading on $UT_m(F)$ is $(e_1,\dots, e_r,0, -e_{r},\dots, -e_{1})$, therefore $E_{r+1,b}$ is homogeneous of degree $e_{2r+2-b}$. Then $g_1=e_{2r+2-b}$, note that $r+1<b<c\leq 2r+1$, therefore $r+1<b<2r+1$ and $2\leq 2r+2-b\leq r$. We also have $g_2=e_{2r+2-c}-e_{2r+2-b}$. Then it follows from (\ref{i5}) that the polynomial $x_{1,0}x_{2,g_1}x_{3,g_2}-x_{1,0}^{\ast}x_{2,g_1}x_{3,g_2}$ lies in $J$. We may use this identity to obtain a monomial that is congruent to $M$ modulo $J$ and satisfies $(I)-(V)$.
 
  \vspace{0,5cm}
		{\bf Condition (VI)}
		\vspace{0,3cm}

Now assume that the hypothesis of $(VI)$ holds, i.e., $k\geq 2$, $g_{k-1} \notin \{e_1,\dots, e_r\}$ and $g_k\in\{e_1,\dots,e_r\}$. Let $a^{\prime}<b^{\prime}<c^{\prime}$ be indices such that $(x_{u_{k-1},g_{k-1}}^{\delta_{k-1}})_S=E_{a^{\prime},b^{\prime}}$ and $(x_{u_{k},g_{k}}^{\delta_{k}})_S=E_{b^{\prime},c^{\prime}}$. Note that $E_{a^{\prime},b^{\prime}}$ and $E_{b^{\prime},c^{\prime}}$ are homogeneous of degrees $g_{k-1}$ and $g_k$, respectively. Since $g_{k-1}\notin \{e_1,\dots, e_r\}$ and $g_k\in\{e_1,\dots,e_r\}$ we conclude that $c^{\prime}=r+1$. Then $g_k=e_{b^{\prime}}$ and $g_{k-1}=e_{a^{\prime}}-e_{b^{\prime}}$, moreover $1\leq a^{\prime}<b^{\prime}\leq r$. Hence it follows from (\ref{i5}) that the polynomial $x_{3,0}x_{2,g_k}x_{1,g_{k-1}}-x_{3,0}^{\ast}x_{2,g_k}x_{1,g_{k-1}}$ lies in $J$. Therefore we conclude that the polynomial $$\left(x_{3,0}x_{2,g_k}^{\ast}x_{1,g_{k-1}}^{\ast}-x_{3,0}^{\ast}x_{2,g_k}^{\ast}x_{1,g_{k-1}}^{\ast}\right)^{\ast}=x_{1,g_{k-1}}x_{2,g_k}x_{3,0}^{\ast}-x_{1,g_{k-1}}x_{2,g_k}x_{3,0}$$ lies in $J$. We use this identity to obtain a monomial that is congruent to $M$ modulo $J$ and that satisfies $(I)-(VI)$.

   \vspace{0,5cm}
		{\bf Condition (VII)}
		\vspace{0,3cm}
 
  	This monomial is good unless $m=2r+1$, $k=1$ and $g_1\in\{e_1,\dots,e_r\}$. Now assume that this is the case, it follows from identities (\ref{i1}) and (\ref{i4}) that $M$ is congruent modulo $J$ to a linear combination of monomials that satisfy $(I)-(VII)$, i.e., $M$ is congruent modulo $J$ to a linear combination of good monomials.	
\end{proof}


\begin{theorem}\label{mainfine}
	Let $G=\mathbb{Z}^{\lfloor \frac{m}{2} \rfloor}$. The polynomials (\ref{i1}), (\ref{i2}), (\ref{i3}), (\ref{i4}) and (\ref{i5}) together with the monomial identities of the form $x_{1,g_1}\cdots x_{n,g_n}$, where $n\leq 3$ and $g_i\neq 0$ for all $i$, form a basis for the $(G,\ast)$-identities for $UT_m(F)$ with the finest grading and reflection involution. Moreover the images of the good monomials of degree $n$, under the canonical quotient map, form a basis for $P_n^{G,\ast}/P_n^{G,\ast}\cap T_{G,\ast}(UT_m(F))$.
\end{theorem}
\begin{proof}
	Let $J$ be the $T_{G,\ast}$-ideal generated by the identities (\ref{i1}), (\ref{i2}), (\ref{i3}), (\ref{i4}) and (\ref{i5}) and the monomials
	\begin{align}\label{monj}
	x_{1,g_1}\cdots x_{n,g_n},
	\end{align}
	in $T_{G,\ast}(UT_m)$ with $g_i\neq 0$ for $i=1,\dots,n$ and $n\leq 3$. Proposition \ref{id} implies that $J\subseteq T_{G,\ast}(UT_m(F))$. Proposition \ref{p3} implies that every monomial in $T_G(UT_m(F))$ lies in $J$. It remains to prove the reverse inclusion. We prove that the good monomials are linearly independent modulo $T_{G,\ast}(UT_m(F))$.	Let $T=\{x_{1,h_1},\dots, x_{n,h_n}\mid h_1,\dots, h_n\in G\}$ be a set of indeterminates. Let $\mathcal{M}(T)$ be the set of good monomials of degree $n$ in the inteterminates in $T$.  We prove that $\mathcal{M}(T)$ is linearly independent modulo $T_{G,\ast}(UT_m(F))$, this implies that the good monomials are independent modulo $T_{G,\ast}(UT_m(F))$. Let $k$ be the number of indices $i$ such that $h_i\neq 0$. Let $M,M^{\prime}\in \mathcal{M}(T)$. If $k>1$ then Lemma \ref{goodsub} implies that if $S$ is an elementary substitution such that $(M)_S=(M^{\prime})_S\neq 0$ then $M=M^{\prime}$, hence the  $\mathcal{M}(T)$ is linearly independent modulo $T_{G,\ast}(UT_m)$ in this case. Now assume that $k=1$. Let 
	\begin{align*}
		M=x_{k_{1,1},0}^{\delta_{1,1}}\cdots x_{k_{1,s_1},0}^{\delta_{1,s_1}}x_{u_1,g_1}^{\delta_1}x_{k_{2,1},0}^{\delta_{2,1}}\cdots x_{k_{2,s_2},0}^{\delta_{2,s_2}}
	\end{align*}
	and let $M^{\prime}=M_1^{\prime}x_{u_1,g_1}^{\epsilon_1}M_2^{\prime}$ where $M_1^{\prime}$ and $M_2^{\prime}$ are monomials in indeterminates of degree $0$. If $(M)_S=(M^{\prime})_S\neq 0$ for an elementary substitution $S$ then we argue as in the proof of Lemma \ref{goodsub} to conclude that $\epsilon_1=\delta_1$ and that $M_1^{\prime}=x_{k_{1,1},0}^{\epsilon_{1,1}}\cdots x_{k_{1,s_1},0}^{\epsilon_{1,s_1}}$, $M_2^{\prime}=x_{k_{2,1},0}^{\epsilon_{2,1}}\cdots x_{k_{2,s_2},0}^{\epsilon_{2,s_2}}$. If $m=2r$ or $m=2r+1$ and $\mathrm{deg}_G\, M, \mathrm{deg}_G\, M^{\prime}\notin\{e_1,\dots,e_r\}$ then $(M)_S=(M^{\prime})_S\neq 0$, for an elementary substitution $S$, also implies that $M=M^{\prime}$, hence $\mathcal{M}(T)$ is linearly independent modulo $T_{G,\ast}(UT_m)$ in this case. Now assume that  $m=2r+1$ and $\mathrm{deg}_G\, M, \mathrm{deg}_G\, M^{\prime}\in\{e_1,\dots,e_r\}$. 
	Note that if $\delta_{i,j}=\ast$ for some $i,j$ then there exists an elementary substitution $S$ such that $(M)_S\neq 0$ and $(M)_S=(M^{\prime})_S$ implies that $M=M^{\prime}$. Indeed, assume first that $i=1$. We have $g_1=\mathrm{deg}_G\, M=\in\{e_1,\dots,e_r\}$. Since $m=2r+1$ the degree of $E_{i,r+1}$ is $e_i$. We consider the substitution $S$ such that
	\begin{align*}
	&(x_{u_1,g_1}^{\delta_1})_S=E_{i,r+1},\\ &(x_{k_{1,j},0}^{\delta_{1,j}})_S=E_{i,i}\mbox{ for }j=1,\dots s_1,\\ &(x_{k_{2,j^{\prime}},0}^{\delta_{2,j^{\prime}}})_S=E_{r+1,r+1}\mbox{ for }j^{\prime}=1,\dots s_2.
	\end{align*}
	Condition (VII) from Definition \ref{defgood} implies that if $M^{\prime}$ is a good monomial in $\mathcal{M}(T)$ such that $(M)_S=(M^{\prime})_S$ then $M=M^{\prime}$. The proof for $i=2$ is analogous. As a consequence $\mathcal{M}(T)$ is linearly independent modulo $T_{G,\ast}(UT_m)$ in this case.
	Now let $f$ be a polynomial in $P_n^{G,\ast}\cap T_{G,\ast}(UT_m(F))$. We may write $f$ modulo $J$ as a linear combination 
	\begin{align*}
	\sum_{i=1}^{t} \lambda_i M_i,
	\end{align*}
	of good monomials $M_1,\dots, M_t$ in $\mathcal{M}(T)$ for some set of indeterminates $T$. Since $\mathcal{M}(T)$ is linearly independent modulo $T_{G,\ast}(UT_m(F))$ we conclude that $\lambda_1=\dots=\lambda_l=0$. Therefore $f\in J$. Since the field is of characteristic zero this implies that $T_{G,\ast}(UT_m(F))\subseteq J$.	
\end{proof}

\begin{remark}
	Let $G=\mathbb{Z}^{r}$ and let $\mathcal{U}$ be the finest $G$-grading on $UT_{2r}(F)$. If $\mathrm{dim}\, \mathcal{U}_g=1$ then there exist $1\leq i<j\leq 2r$ with $i+j=2r+1$ such that $\mathcal{U}_g=\mathrm{span}\, \{E_{i,j}\}$. Note that in this case $i\leq r<j$, hence $E_{i,j}^{s}=-E_{i,j}$. Therefore $x_{1,g}^{\ast}+x_{1,g}$ is a $(G,\ast)$-identity for $\mathcal{U}$ with the symplectic involution.
\end{remark}

An adaptation of the argument above proves the following result.

\begin{theorem}
Let $G=\mathbb{Z}^{r}$ and let $\mathcal{U}$ be the finest $G$-grading on $UT_{2r}(F)$. The $(G,\ast)$-identities of the algebra $UT_{2r}(F)$ with the finest grading and symplectic involution follow from the identities
\begin{align}
&[x_{1,0},x_{2,0}],\\
&x_{1,g}^{\ast}+x_{1,g},\,  \mathrm{dim}\, \mathcal{U}_g=1,
\end{align}
together with the monomial identities of the form $x_{1,g_1}\cdots x_{n,g_n}$, where $n\leq 3$ and $g_i\neq 0$ for all $i$. Moreover the the images of the good monomials of degree $n$, under the canonical quotient map, form a basis for $P_n^{G,\ast}/P_n^{G,\ast}\cap T_{G,\ast}(UT_m)$.
\end{theorem}

We now determine the asymtotic growth for the $(G,\ast)$-codimensions of $UT_m(F)$ with the finest grading and reflection involution. We begin by determining the exact formula for $(G,\ast)$-codimensions of $UT_2(F)$ and $UT_3(F)$. The following easy remark will be useful in the proof of the next results.

\begin{remark}\label{km1}
An elementary matrix of degree different from $0$ in $UT_m(F)$ with the finest grading lies in the Jacobson radical of $UT_m(F)$, therefore the product of $m$ such matrices is zero. As a consequence in the good monomial in \ref{defgood} the number $k$ of indeterminates of degree different from $0$ is at most $m-1$.
\end{remark}

\begin{proposition}\label{cod2,3}
	We have
	\begin{align*}
	c_n^{G,\ast}(UT_2(F))=(n+2)2^{n-1},
	\end{align*}
 for $n\geq 1$. Moreover
	\begin{align*}
	c_n^{G,\ast}(UT_3(F))=2n(n+5)3^{n-2}-(n-2)2^{n-1},
	\end{align*}		
 for $n\geq 1$.
\end{proposition}
\begin{proof}
 We compute the $(G,\ast)$-codimensions for $UT_3(F)$, the proof that $c_n^{G,\ast}(UT_2(F))=(n+2)2^{n-1}$ for $n\geq 1$ is analogous and simpler so we omit it. Theorem \ref{mainfine} implies that $c_n^{G,\ast}(UT_3(F))$ is the number of good monomials in $P_n^{G,\ast}$. Remark \ref{km1} implies that the number of indeterminates of degree different from $0$ in a good monomial is at most $2$. Clearly there are $2^n$ good monomials with indeterminates of degree $0$ only. Next we compute the number of good monomials with $k=1$ indeterminate of degree different from $0$. Let 
 \begin{align}\label{goodk1}
M_1x_{u_1,g_1}^{\delta_1}M_2,
\end{align}
be a good monomial in $P_n^{G,\ast}$, here $g_1\neq 0$ and
\begin{align*}
M_i=x_{k_{i,1},0}^{\delta_{i,1}}\cdots x_{k_{i,s_i},0}^{\delta_{i,s_i}},
\end{align*}
for $i=1,2$. There are two cases to consider.

\vspace{0,3cm}
{\bf Case I:} $g_1=1$
\vspace{0,1cm}

In this case Condition VII in Definition \ref{defgood} implies that  $\delta_{1,1}=\cdots = \delta_{1,s_1}=\emptyset$ or $\delta_{2,1}=\cdots=\delta_{2,s_2}=\emptyset$. We compute the number of such good monomials with $\delta_{1,1}=\cdots = \delta_{1,s_1}=\emptyset$. There are $2$ choices for $\delta_1$, $n$ choices for $u_1$ and for $s_1, s_2$ given there are ${n-1\choose s_1}$ choices for $k_{1,1},\dots, k_{1,s_1};\, k_{2,1},\dots, k_{2,s_2}$, finally there are $2$ choices for each $\delta_{2,j}$, $1\leq j \leq s_2$. Therefore the total number of good monomials in this case is \[\sum_{s_1+s_2=n-1}2n{n-1\choose s_1}2^{s_2}=2n3^{n-1}.\]
Analogously there are $2n3^{n-1}$ good monomials in this case with $\delta_{2,1}=\cdots=\delta_{2,s_2}=\emptyset$. Moreover the number of good monomials in this case with $\delta_{1,1}=\cdots = \delta_{1,s_1}=\emptyset=\delta_{2,1}=\cdots=\delta_{2,s_2}$ is \[\sum_{s_1+s_2=n-1}2n{n-1\choose s_1}=2n2^{n-1}.\] Therefore the total number of good monomials in Case I is $4n3^{n-1}-2n2^{n-1}$.

\vspace{0,3cm}
{\bf Case II:} $g_1=2$
\vspace{0,1cm}

In this case Condition II in Definition \ref{defgood} implies that $M_1=1$ and $\delta_1=\emptyset$. There are $n$ choices for $u_1$ and $2$ choices for each $\delta_{2,j}$, $1\leq j \leq s_2=n-1$. Therefore the number of good monomials in this case is $n2^{n-1}$.

We conclude from the arguments above that the number of good monomials in (\ref{goodk1}) is $4n3^{n-1}-n2^{n-1}$.

It remains to compute the number of good monomials with $k=2$ indeterminates of degree different from $0$. In this case there are $2$ indeterminates of degree $1$. Let 
 \begin{align}\label{goodk2}
M_1x_{u_1,1}^{\delta_1}M_2x_{u_2,1}^{\delta_2}M_3,
\end{align}
be a good monomial in $P_n^{G,\ast}$, here
\begin{align*}
M_i=x_{k_{i,1},0}^{\delta_{i,1}}\cdots x_{k_{i,s_i},0}^{\delta_{i,s_i}},
\end{align*}
for $i=1,2,3$. Conditions III and IV in Definition \ref{defgood} imply that $M_1=1$, $\delta_{2,1}=\cdots = \delta_{2,s_2}=\emptyset$ and $u_1<u_2$. There are $2$ choices for $\delta_1$ and $2$ choices for $\delta_2$. There are ${n\choose 2}$ choices the pair $(u_1,u_2)$. For $s_2, s_3$ given there are ${n-2\choose s_2}$ choices for the indices $k_{2,1},\dots, k_{2,s_2};\, k_{3,1},\dots, k_{3,s_3}$. Finally there are $2$ choices for $\delta_{3,j}$, $j=1,\dots, s_3$. Hence the total number of good monomials in (\ref{goodk2}) is \[\sum_{s_2+s_3=n-2}4{n\choose 2}{n-2\choose s_2}2^{s_3}=2n(n-1)3^{n-2}.\]

Hence we conclude that \[c_n^{G,\ast}(UT_3(F))=2^n+(4n3^{n-1}-n2^{n-1})+2n(n-1)3^{n-2},\] and this yields the desired formula.

\end{proof}

\begin{lemma}\label{nm1}
The number of good monomials in $P_n^{G,\ast}$ with $m-1$ indeterminates of degree different from $0$ is \[2^{\lfloor \frac{m-1}{2}\rfloor}\frac{n!}{(n-m+1)!}m^{n-m+1},\] for every $n\geq m-1$.
\end{lemma}
\begin{proof}

Let 
\begin{align}\label{goodnkl}
M_1x_{u_1,g_1}^{\delta_1}M_2x_{u_2,g_2}^{\delta_2}\cdots M_kx_{u_{m-1},g_{m-1}}^{\delta_{m-1}}M_{m},
\end{align}
be a good monomial in $P_n^{G,\ast}$, here $g_1,\dots,g_m\neq 0$ and
\begin{align*}
M_i=x_{k_{i,1},0}^{\delta_{i,1}}\cdots x_{k_{i,s_i},0}^{\delta_{i,s_i}}.
\end{align*}

Note that the only product of $m-1$ elementary matrices of degree different from $0$ that does not result in $0$ is $$E_{1,2}E_{2,3}\cdots E_{m-1,m}.$$ Hence in this case $g_i$ is the degree of $E_{i,i+1}$ for $i=1,\dots, m-1$. We consider two cases.

{\vspace{0,3cm}}
{\bf Case I: $m=2r+1$}
\vspace{0,1cm}
 
In this case Conditions III and IV in Definition \ref{defgood} imply that $M_1=\cdots=M_r=1$, the indeterminates in $M_{r+1}$ have no $\ast$ and $u_i<u_{2r+1-i}$ for $i=1,\dots, r$. We count the number of monomials in (\ref{goodnkl}) for $s_{r+1},\dots, s_{2r+1}$ fixed. The number of choices for $u_1,\dots, u_{2r}$ is $\frac{n!}{2^r(n-2r)!}$ and the number of choices for $\delta_1,\dots, \delta_{2r}$ is $2^{2r}$. For the indices $$k_{i,1},\dots, k_{i,s_i}\,\,, i=r+1,\dots, 2r+1,$$ the number of choices is equal to $${n-2r\choose s_{r+1},s_{r+2},\dots, s_{2r+1}}.$$ There are $2$ choices for each $\delta_{i,j}$, $1\leq j \leq s_i$, $i=r+2, r+3,\dots, 2r+1$. Hence the number of good monomials in (\ref{defgood}) in this case for $s_{r+1},\dots, s_{2r+1}$ fixed is \[ \frac{n!}{2^{r}(n-2r)!}\cdot 2^{2r}\cdot {n-2r\choose s_{r+1},s_{r+2},\dots, s_{2r+1}}2^{s_{r+2}}2^{s_{r+3}}\cdots 2^{s_{2r+1}}.\]

Therefore the number of good monomials in (\ref{goodnkl}) equals
\begin{align*}
\frac{2^rn!}{(n-2r)!}\sum_{s_{r+1}+\cdots +s_{2r}=n-2r}{n-2r\choose s_{r+1},s_{r+2},\dots, s_{2r+1}}2^{s_{r+2}}2^{s_{r+3}}\cdots 2^{s_{2r+1}}=\frac{2^{r}n!}{(n-2r)!}(2r+1)^{n-2r}.
\end{align*}

{\vspace{0,3cm}}
{\bf Case II: $m=2r$}
\vspace{0,1cm}

For $m=2r$ we argue analogously, in this case Conditions II and III in Definition \ref{defgood} imply that $M_1=\cdots=M_r=1$, $\delta_r=\emptyset$ and $u_i<u_{2r-i}$ for $i=1,\dots, r-1$, hence the total number of good monomials in (\ref{goodnkl}) in this case is 
\begin{align*}
&\frac{2^{2r-2}n!}{2^{r-1}(n-2r+1)!}\sum_{s_{r+1}+\cdots +s_{2r-1}=n-2r+1}{n-2r+1\choose s_{r+1},s_{r+1},\dots, s_{2r}}2^{s_{r+1}}2^{s_{r+2}}\cdots 2^{s_{2r}}\\&=\frac{2^{r-1}n!}{(n-2r+1)!}(2r)^{n-2r+1}.
\end{align*}
\end{proof}

\begin{lemma}\label{boundnk}
Let $N_k(n)$ be the number of good monomials, for $UT_m(F)$, in $P_n^{G,\ast}$ with $k$ indeterminates of degree different from $0$. If $m\geq 4$ then \[N_k(n)\leq (2C)^k\frac{n!}{(n-k)!}m^{n-k},\] for every $n\geq k$, where $C$ is the number of elements in the support of the finest grading of $UT_m(F)$.
\end{lemma}

\begin{proof}

The inequality above clearly holds for $k=0$ and since $m\geq 4$ it also holds for $k=1$. Henceforth we assume that $k\geq 2$.

 Let ${\bf g}=(g_1,\dots, g_k)$ be a $k$-tuple of elements of $G=\mathbb{Z}^{\lfloor \frac{m}{2} \rfloor}$ and let $N_{\bf g}(n)$ be the number of good monomials in $P_n^{G,\ast}$ of the form
\begin{align}\label{goodnklg}
M_1x_{u_1,g_1}^{\delta_1}M_2x_{u_2,g_2}^{\delta_2}\cdots M_kx_{u_{k},g_{k}}^{\delta_{k}}M_{k+1},
\end{align}
where 
\begin{align*}
M_i=x_{k_{i,1},0}^{\delta_{i,1}}\cdots x_{k_{i,s_i},0}^{\delta_{i,s_i}}.
\end{align*}

The result follows if we prove that 
\begin{align}\label{toprove}
N_g(n)\leq 2^k\frac{n!}{(n-k)!}m^{n-k},
\end{align}
for $n\geq k$.

We assume that the monomial  in (\ref{goodnklg}) is not an identity for $UT_m(F)$, hence Corollary \ref{Mlinear} implies that 
\begin{align*}
M=x_{u_1,g_1}x_{u_2,g_2}\cdots x_{u_{k},g_{k}},
\end{align*}
is not an identity for $UT_m(F)$. Recall that $k\geq 2$, therefore Lemma \ref{one} implies that there exists one substitution $S=(E_{i_1,i_2}, E_{i_2,i_3},\dots, E_{i_k, i_{k+1}})$ such that $M_S\neq 0$. Then $g_l=\mathrm{deg}\, E_{i_l,i_{l+1}}$ for $l=1,\dots, k$. We denote by ${\bf i}$ the sequence $i_1<i_2<\cdots < i_k<i_{k+1}$ obtained from the $k$-tuple ${\bf g}$ in this way. 

Now we establish the notation, based on the sequence ${\bf i}$, that will be needed in the proof of the lemma. Let $t$ denote the number of elements in the set $I$ of indices $a$ such that there exists $b$ with $1\leq a<b\leq k+1$ and $i_a+i_b=m+1$, moreover let $j_1<j_2<\cdots < j_{k+1-t}$ be the indices in the complement of $I$. Note that for such an $a$ it follows that $g=g_a+g_{a+1}+\cdots +g_{b-1}$ is the degree of the matrix $E_{i_a, i_b}$, since $E_{i_a, i_b}^{\ast}=E_{i_a, i_b}$ we conclude that the homogeneous component indexed by $g$ has dimension $1$. Then Conditions II and III in Definition \ref{defgood} imply that $M_a=1$ in (\ref{goodnklg}). Henceforth we denote $w({\bf i}):=k+1-t$ and refer to this number as the weight of ${\bf i}$. 

We obtain an upper bound for $N_{\bf g}(n)$ by counting the number of monomials in (\ref{goodnklg}) such that $M_a=1$ whenever $a\in I$. The number of choices for $u_1,\dots, u_k$ is $\frac{n!}{(n-k)!}$, for $\delta_1,\dots, \delta_k$ there are $2^k$ choices. The indices $k_{j_l,1},\dots,  k_{j_l, s_{j_l}}$, $l=1,\dots, k+t-1$ can be chosen in ${n-k\choose s_{j_1},\dots, s_{j_{k+1-t}}}$ ways and there are $2^{s_{j_l}}$ choices for $\delta_{j_l,1},\dots, \delta_{j_l,s_{j_l}}$, $j=1,\dots, l$. Therefore the number good of monomials in this case for $s_{j_1},\dots, s_{j_{k+1-t}}$ given is at most \[2^k\frac{n!}{(n-k)!}{n-k\choose s_{j_1},\dots, s_{j_{k+1-t}}}2^{s_{j_1}}2^{s_{j_2}}\cdots 2^{s_{j_{k+1-t}}}.\] Therefore 

\begin{align}\label{ineqng}
\begin{split}
N_{\bf g}(n)&\leq 2^k\frac{n!}{(n-k)!}\sum_{s_{j_1}+\cdots + s_{j_{k+1-t}}=n-k}{n-k\choose s_{j_1},\dots, s_{j_{k+1-t}}}2^{s_{j_1}}2^{s_{j_2}}\cdots 2^{s_{j_{k+1-t}}}\\&=2^k\frac{n!}{(n-k)!}(2(k+1-t))^{n-k}.
\end{split}
\end{align}

Note that adding one entry to a sequence does not decrease its weight, therefore $w({\bf i})$ is not greater than the weight of the sequence $1<2<\cdots < m$. Hence we conclude that $w({\bf i})\leq r+1$, where $r=\lfloor \frac{m}{2} \rfloor$. If $w({\bf i})\leq r$ then inequality (\ref{ineqng}) implies that \[N_{\bf g}(n)\leq 2^k\frac{n!}{(n-k)!}(2r)^{n-k}\leq2^k\frac{n!}{(n-k)!}m^{n-k},\] hence inequality (\ref{toprove}) holds in this case.

Now assume that $w({\bf i})=r+1$, we claim that $m=2r+1$ and $r+1$ is an entry in ${\bf i}$. Indeed if $m=2r$ the weight of $1,2,\dots, 2r$ is $r$, hence $w({\bf i})\leq r$. This implies that $m=2r+1$. The sequence without $r+1$ with greatest weight is $1, 2,\cdots, r, r+2, r+3,\dots, 2r+1$, the weight of this sequence is $r$. Therefore $w({\bf i^{\prime}})\leq r$ if $r+1$ is not an entry in ${\bf i^{\prime}}$. Thus $r+1$ is an entry in ${\bf i}$. Then Conditions IV-VI in Definition \ref{defgood} imply that $\delta_{j_w,1}=\cdots =\delta_{j_w,s_w}=\emptyset$ for some $w$. Therefore
\begin{align*}
N_{\bf g}(n)&\leq  2^k\frac{n!}{(n-k)!}\sum_{s_{j_1}+\cdots + s_{j_{r+1}}=n-k}{n-k\choose s_{j_1},\dots, s_{j_{r+1}}}2^{s_{j_1}}\cdots 2^{s_{j_{w-1}}}2^{s_{j_{w+1}}}\cdots 2^{s_{r+1}}\\&=2^k\frac{n!}{(n-k)!}(2r+1)^{n-k},
\end{align*}
hence (\ref{toprove}) also holds in this case.
\end{proof}

\begin{theorem}\label{expo}
	For $UT_m(F)$, $m\geq 2$, with the finest grading and reflection involution $\ast$ we have 
	\begin{align*}
	c_n^{G,\ast}(UT_m(F))\sim K_m\, n^{m-1}m^n,
	\end{align*}
 where $K_m=\frac{2^{\lfloor (m-1)/2 \rfloor}}{m^{m-1}}$.	In particular,
	\begin{align*}
	\mathrm{exp}^{G,\ast}(UT_m(F))= m.
	\end{align*}		
\end{theorem}
\begin{proof}
Theorem \ref{mainfine} implies that $c_n^{G,\ast}(UT_m(F))$ is the number of good monomials in $P_n^{G,\ast}$. Let $N_k(n)$ be the number of good monomials in $P_n^{G,\ast}$ in $k$ indeterminates of degree different from $0$, then Remark \ref{km1} implies that
\begin{align*}
c_n^{G,\ast}(UT_m(F))=N_0(n)+\cdots+N_{m-1}(n).
\end{align*}

The result for $m=2,3$ follows from Proposition \ref{cod2,3}. Henceforth we assume that $m\geq 4$.

Lemma \ref{boundnk} implies that 
\[\lim_{n\to\infty}\frac{N_k(n)}{n^{m-1}m^n}=0\] for $k<m-1$.  Lemma \ref{nm1} implies that \[N_{m-1}(n)=2^{\lfloor \frac{m-1}{2}\rfloor}\frac{n!}{(n-m+1)!}m^{n-m+1},\] for $n\geq m-1$. Therefore we conclude that \[c_n^{G,\ast}(UT_m(F))\sim N_{m-1}(n)\sim\left(\frac{2^{\lfloor\frac{m-1}{2}}\rfloor}{m^{m-1}}\right)n^{m-1}m^n.\]
\end{proof}

An adaptation of the proof of the theorem above yields the following result.

\begin{theorem}\label{expo1}
	For $UT_{2r}(F)$, with the finest grading and symplectic involution $s$ we have 
	\begin{align*}
	c_n^{G,\ast}(UT_{2r}(F))\sim \left(\frac{1}{2^rr^{2r-1}}\right)n^{2r-1}(2r)^n
	\end{align*}
	In particular,
	\begin{align*}
	\mathrm{exp}^{G,\ast}(UT_{2r}(F))= 2r.
	\end{align*}	
\end{theorem}

As a consequence of Proposition \ref{coars}, Theorem \ref{expo} and Theorem \ref{expo1} we obtain the following result.

\begin{corollary}
	For an arbitrary $G$-grading on $UT_m(F)$ and graded involution $\ast$ we have $\mathrm{exp}^{G,\ast}(UT_m)=m$.
\end{corollary}

\section{Graded Identities with Involution for $UT_3(F)$}\label{s3}
In this section we describe the gradings on $UT_3(F)$ that admit graded involution. If the group $G$ does not have elements of order $2$ then the only two possible gradings (up to equivalence) are the trivial grading and the grading in the previous section. If $G$ has elements of order $2$ then it admits $G$-gradings equivalent to the $\mathbb{Z}_2$-grading on $UT_3(F)$ induced by $(0,1,0)$. Such gradings admit as a graded involution, up to equivalence, the reflection involution. We describe the $\mathbb{Z}_2$-graded identities with involution in this case. 

Let $Y^{+}=\{y_1^{+},y_2^{+},\dots\}$, $Y^{-}=\{y_1^{-},y_2^{-},\dots\}$, $Z^{+}=\{z_1^{+},z_2^{+},\dots\}$ and $Z^{-}=\{z_1^{-},z_2^{-},\dots\}$ be four infinite countable and pairwise disjoint sets and let $Y=Y^{+}\cup Y^{-}$, $Z=Z^{+}\cup Z^{-}$. Let $F\langle Y\cup Z, \ast \rangle$ be the free algebra freely
generated by $Y\cup Z$ with the $\mathbb{Z}_2$-grading such  that the indeterminates in $Y$ have degree $0$ and the indeterminates in $Z$ have degree $1$ and involution $\ast$ such that the indeterminates in $Y^{+}, Z^{+}$ are symmetric and the indeterminates in $Y^{-},Z^{-}$ are skew-symmetric. Then $F\langle Y\cup Z, \ast \rangle$ is the free $\mathbb{Z}_2$-graded algebra with graded involution $\ast$.

Henceforth the commutators of higher order are left-normed, i.e., if $a_1,\dots, a_n$ are elements of an associative algebra then $[a_1,\dots, a_n]:=[[a_1,\dots, a_{n-1}], a_n]$ for $n\geq 3$.  A polynomial in $F\langle Y\cup Z, \ast \rangle$ is $Y^{+}$-proper if it is a linear combination of polynomials of the form
\begin{align*}
(y_1^{-})^{q_1}\cdots (y_b^{-})^{q_b}(z_1^{+})^{r_1}\cdots (z_c^{+})^{r_c}(z_1^{-})^{s_1}\cdots (z_d^{-})^{s_d}u_1^{t_1}\cdots u_e^{t_e},
\end{align*}
where $q_i,r_i,s_i,t_i$ are non-negative integers and $u_1,\dots, u_n$ are left-normed commutators (of degree $\geq 2$) of indeterminates in $Y\cup Z$. Let $\Gamma$ be the subspace of $F\langle Y\cup Z, \ast \rangle$ of $Y^{+}$-proper polynomials. As in the case of (ungraded) algebras with involution, since $\mathrm{char}\, F=0$, any $T_{\mathbb{Z}_2,\ast}$-ideal is generated by its $Y^{+}$-proper multilinear polynomials, the proof follows closely the one for algebras with involution (see \cite[Lemma 2.1]{DG}) so we omit it.
We recall that $P_n^{\mathbb{Z}_2,\ast}$ is the subspace of $F\langle Y\cup Z\rangle$ of the multilinear polynomials of degree $n$ in the indeterminates in $Y\cup Z$ such that in each monomial the index $i$ appears once for $i=1,\dots,n$, we also define $\Gamma_n=\Gamma \cap P_n^{\mathbb{Z}_2,\ast}$. 
 Let $A$ be a $\mathbb{Z}_2$-graded algebra with graded involution. We define $\Gamma_n(A)=\Gamma_n/\Gamma_n\cap T_{\mathbb{Z}_2,\ast}(A)$, the $n$-th $Y^{+}$-proper codimension of $A$ is $\mathrm{dim}\, \Gamma_n(A)$ and is denoted by $\gamma_n^{\mathbb{Z}_2,\ast}(A)$. In the next result we provide a relation between the proper and ordinary codimensions of $A$. The proof follows closely the one in the case of (ungraded) algebras with involution (see \cite[Theorem 2.3]{DG}) so we omit it.

\begin{proposition}\label{codim}
	Let $A$ be a $\mathbb{Z}_2$-graded algebra with graded involution $\ast$. The codimension sequence $c_n^{\mathbb{Z}_2,\ast}(A)$ and the proper codimension sequence $\gamma_n^{\mathbb{Z}_2,\ast}(A)$ are related by the following equalities
	\begin{align*}
	c_n^{\mathbb{Z}_2,\ast}(A)=\sum_{i=0}^{n}{n\choose i}\gamma_n^{\mathbb{Z}_2,\ast}(A).
	\end{align*}
\end{proposition}

In this section we consider $\mathcal{U}=UT_3(F)$ with the $\mathbb{Z}_2$-grading induced by $(0,1,0)$ and the reflection involution. Then 
\begin{align*}
&\mathcal{U}_0^{+}=\langle E_{1,1}+E_{3,3}, E_{2,2}, E_{1,3}\rangle,\ \ \mathcal{U}_0^{-}=\langle E_{1,1}-E_{3,3} \rangle\\
&\mathcal{U}_1^{+}=\langle E_{1,2}+E_{2,3}\rangle,\ \ \mathcal{U}_1^{-}=\langle E_{1,2}-E_{2,3}\rangle.
\end{align*}

\begin{lemma}\label{ids}
	The polynomials
	\begin{align}
	&x_1x_2x_3-x_3x_2x_1,\, x_1,x_3 \mbox{ both lie in one of the subspaces } Y^{-}, Z^{+}, Z^{-}\label{1}\\
	&z_1z_2z_3\label{2}\\
	 &[x_1,x_2],\,x_1,x_2\mbox{ both lie in one of the subspaces } Y^{+},Y^{-}, Z^{+}, Z^{-}\label{3}\\ 
	 &[y_1^{+},y_2^{-}][x_3,x_4]\label{4}\\
	 &[y_1^{+},z_1z_2]\label{6}\\
	 &[y_1^+,y_2^-]z_3, z_3[y_1^+,y_2^-]\label{7}\\
	 & [y_1^{+},y_2^{-}]y_3^{-}+y_3^{-}[y_1^{+},y_2^{-}]\label{9}\\
	 & y_1^{-}z_1y_2^{-}\label{10}\\
	 & z_1y_1^{-}z_2\label{11}\\
	 & z_1^{+}z_2^{-}+z_{2}^{-}z_1^{+}\label{12}\\
	 & z_1z_2y_1^{-}+y_1^{-}z_1z_2\label{13}
	\end{align} 
	are graded identities with involution for $UT_3(F)$.
\end{lemma}
\begin{proof}
The proof is an easy verification that the result of every elementary substitution is zero, so we omit it.
\end{proof}

\begin{theorem}\label{basis}
	The polynomials (\ref{1})-(\ref{13}) form a basis for the $(\mathbb{Z}_2,\ast)$-identities with for $\mathcal{U}=UT_3(F)$ with the elementary $\mathbb{Z}_2$-grading induced by $(0,1,0)$ and the reflection involution. 
\end{theorem}
\begin{proof}
Let $I$ be the $T_{(\mathbb{Z}_2,\ast)}$-ideal generated by the polynomials (\ref{1})-(\ref{13}). Lemma \ref{ids} implies that $I\subseteq T_{\mathbb{Z}_2,\ast}(\mathcal{U})$. Since the field is of characteristic zero to prove that $I= T_{\mathbb{Z}_2,\ast}(\mathcal{U})$ it is sufficient to prove that $\Gamma_n\cap T_{\mathbb{Z}_2,\ast}(\mathcal{U})\subseteq I$. Let $f$ be a polynomial in $\Gamma_n\cap T_{\mathbb{Z}_2,\ast}(\mathcal{U})$. If the number of indeterminates in $f$ of odd degree is $>2$ then $f$ is a consequence of (\ref{2}), hence $f$ lies in $I$.  The rest of the proof is divided in three cases.

\vspace{0,3cm}

{\bf Case 1:} $f$ is a polynomial in even indeterminates only.

Note that if every indeterminate in $f$ is skew-symmetric then $f$ is congruent modulo $I$ to a scalar multiple of 
\begin{align}\label{y-}
y_1^{-}\cdots y_n^{-}.
\end{align}
Let $\alpha$ be the scalar such that $f\equiv_I\alpha y_1^{-}\cdots y_n^{-}$. Since $f\in T_{\mathbb{Z}_2,\ast}(\mathcal{U})$ and $I\subseteq  T_{\mathbb{Z}_2,\ast}(\mathcal{U})$ we conclude that $\alpha y_1^{-}\cdots y_n^{-} \in T_{\mathbb{Z}_2,\ast}(\mathcal{U})$. The monomial (\ref{y-}) is not an identity for $\mathcal{U}$, hence we conclude that $\alpha=0$. Therefore $f\in I$.

Now we assume that $f$ has symmetric indeterminates. Since $f$ is $Y^{+}$-proper, if it is a polynomial in symmetric even indeterminates only then it is a consequence of (\ref{3}), therefore it lies in $I$. Hence we assume that $f$ has symmetric and skew-symmetric indeterminates. Let $y_{i_1}^{-},\dots, y_{i_a}^{-}, y_{j_1}^{+},\dots, y_{j_b}^{+}$, $a,b>0$, be the indeterminates in $f$. Recall that $f$ is a linear combination of polynomials of the form
\begin{align}\label{poly}
y_{k_1}^{-}\cdots y_{k_r}^{-}u_1\cdots u_s,
\end{align}
where $u_1,\dots, u_s$ are commutators. Note that if in a commutator $u_i$ the first two indeterminates are both symmetric or both skew-symmetric then $u_i$ is a consequence of (\ref{3}). Hence in this case the polynomial (\ref{poly}) lies in $I$. Otherwise we may use the anti-commutativity in the first two indeterminates of the left-normed commutator and the Jacobi identity to write $f$ as a linear combination of polynomials in (\ref{poly}) such that the first indeterminate in $u_i$ is symmetric for $i=1,\dots,s$. If $s>1$ then the polynomial (\ref{poly}) is a consequence of (\ref{4}), hence $f$ is congruent modulo $I$ to a linear combination of polynomials the form (\ref{poly}) with $s=1$ and first indeterminate symmetric in the commutator. Note that $[y_{i_1}^{+},y_{i_2}^{-},\dots, y_{i_k}^{-}]$ is a symmetric polynomial of degree $0$,  hence (\ref{3}) implies that $[y_{i_1}^{+},y_{i_2}^{-},\dots, y_{i_k}^{-},y_{i_{k+1}}^{+}]$ lies in $I$. Therefore $f$ is congruent modulo $I$ to a linear combination of polynomials of the form 
\begin{align}\label{pol}
y_{k_1}^{-}\cdots y_{k_r}^{-}[y_{j_1}^{+},y_{t_{1}}^{-},\dots, y_{t_s}^{-}].
\end{align}
We use the polynomial in (\ref{9}) to conclude that $$[y_{j_1}^{+},y_{t_{1}}^{-},\dots, y_{t_s}^{-}]\equiv_I(-2)^{s-1}y_{t_s}^{-}\cdots y_{t_2}^{-}[y_{j_1}^{+},y_{t_1}^{-}].$$ 
Now note that in the polynomial above the skew-symmetric indeterminates may be reordered modulo $I$ due to (\ref{1}) and (\ref{3}). Hence every polynomial in $\Gamma_n$ is congruent modulo $I$ to a linear combination of polynomials of the form
\begin{align}\label{basiseven}
p_{j}=y_{i_1}^{-}\cdots y_{i_{n-2}}^{-}[y_{j}^{+},y_{i_{n-1}}^{-}],
\end{align}
where $i_1<\cdots < i_{n-1}, \{i_1,\dots, i_{n-1},j\}=\{1,\dots,n\}$.
We claim that the polynomials $p_1,\dots, p_n$ are linearly independent modulo $T_{\mathbb{Z}_2,\ast}(\mathcal{U})$. Indeed let $\alpha_1,\dots, \alpha_n$ be scalars such that 
\begin{align*}
\alpha_1p_1+\cdots+\alpha_np_n\in T_{\mathbb{Z}_2,\ast}(\mathcal{U}).
\end{align*}
Since different sets of indeterminates appear in different polynomials in (\ref{basiseven}) we conclude that $\alpha_ip_i\in T_{\mathbb{Z}_2,\ast}(\mathcal{U})$. Note that $p_i\notin T_{\mathbb{Z}_2,\ast}(\mathcal{U})$, $i=1,\dots, n$, therefore we conclude that $\alpha_1,\dots, \alpha_n=0$. Now let $\alpha_1^{\prime},\dots, \alpha_n^{\prime}$ be scalars such that
\begin{align*}
f\equiv_I \alpha_1^{\prime}p_1+\cdots + \alpha_n^{\prime}p_n,
\end{align*}
Then $\alpha_1^{\prime}p_1+\cdots + \alpha_n^{\prime}p_n$ is a $(\mathbb{Z}_2,\ast)$-identity for $\mathcal{U}$, as a consequence $\alpha_1^{\prime}=\dots = \alpha_n^{\prime}=0$. Therefore $f\in I$.

\vspace{0,3cm}

{\bf Case 2:} $f$ is a polynomial in one odd indeterminate.

As a consequence of (\ref{3}), (\ref{4}) and (\ref{7}) the polynomial $f$ is congruent modulo $I$ a linear combination of polynomials of the form 
\begin{align}\label{polyzbas1}
y_{i_1}^{-}\cdots y_{i_{n-1}}^{-}z_j,
\end{align}
where $i_1<\cdots< i_{n-1}$, $\{i_1,\dots, i_{n-1},j\}=\{1,\dots, n\}$ and of polynomials in $\Gamma_n$ of the form
\begin{align}\label{polyz2}
g=y_{i_1}^{-}\cdots y_{i_k}^{-}[z_j,y_{j_1},\cdots, y_{j_{n-k}}].
\end{align}
It follows from the identities (\ref{3}) and (\ref{7}) that we may reorder, modulo $I$, the even indeterminates in the commutator in $g$ and assume that the symmetric ones appear before the skew-symmetric ones. We rewrite (\ref{10}) as $y_1^{-}y_2^{-}z_1+y_1^{-}[z_1,y_2^{-}]$, therefore modulo $I$ we may assume that, if $k\geq 1$, the even skew-symmetric indeterminates of $g$ are outside of the commutator. As a consequence $g$ is congruent, modulo $I$, to 
\begin{align}\label{polyzbas2}
y_{i_1}^{-}\cdots y_{i_a}^{-}[z_j,y_{k_1}^{+},\cdots, y_{k_b}^{+}], 
\end{align}
where $b>0$, $i_1<\cdots<i_a$, $k_1<\cdots<k_b$, $\{i_1,\dots, i_a,j,k_1,\dots,k_b\}=\{1,\dots, n\}$
or
\begin{align}\label{polyzbas3}
[z_j,y_{k_1}^{+},\cdots, y_{k_b}^{+},y_{i_1}^{-},\cdots, y_{i_a}^{-}],
\end{align}
where $a>0$, $i_1<\cdots<i_a$, $k_1<\cdots<k_b$, $\{i_1,\dots, i_a,j,k_1,\dots,k_b\}=\{1,\dots, n\}$.

Hence $f$ is congruent modulo $I$ to a linear combination of the polynomials in (\ref{polyzbas1}), (\ref{polyzbas2}) and (\ref{polyzbas3}). It is easy to verify that these polynomials are linearly independent modulo $T_{\mathbb{Z}_2,\ast}(\mathcal{U})$. Hence we conclude, as in the previous case, that $f\in I$.

\vspace{0,3cm}

{\bf Case 3:} $f$ is a polynomial in two odd indeterminates.

Note that if no symmetric even indeterminate appears in $f$ then it follows from identities (\ref{3}), (\ref{11}), (\ref{12}) and (\ref{13}) that $f$ is congruent modulo $I$ to a linear combination of the polynomials
\begin{align}\label{basis31}
y_{i_1}^{-}\cdots y_{i_{n-2}}^{-}z_{j_1}z_{j_2},
\end{align}
where $i_1<\cdots< i_{n-2}$, $j_1<j_2$ and $\{i_1,\dots, i_{n-2},j_1,j_2\}=\{1,\dots,n\}$.

Now assume that $f$ has even symmetric indeterminates. Let $u=[z_{j_1},y_{t_1},\cdots, y_{t_r},z_{j_2}, y_{s_1},\cdots, y_{s_{r^{\prime}}}]$, $r,r^{\prime}\geq 0$, be a commutator involving both odd indeterminates. We claim that, if at least one even indeterminate is symmetric, then $u$ is a linear combination of $Y^{+}$-proper polynomials in which the two odd indeterminates do not appear in the same commutator. Indeed, if one of the indeterminates $y_{s_i}$, $1\leq i \leq r^{\prime}$ is symmetric then it follows from (\ref{6}) that $u\in I$. Now assume that these indeterminates are skew-symmetric. In this case it follows from (\ref{13}) that $u$ is congruent modulo $I$ to a scalar multiple of 
\begin{align*}
y_{s_{r^{\prime}}}\cdots y_{s_1} u^{\prime},
\end{align*}
where $u^{\prime}=[z_{j_1},y_{t_1},\cdots, y_{t_r},z_{j_2}]$.  We use identities (\ref{11}) and (\ref{13}) to conclude that $u^{\prime}$ is congruent, modulo $I$, to a scalar multiple of a polynomial of the form
\begin{align*}
y_{i_1}^{-}\cdots y_{i_a}^{-}[z_{j_1},y_{k_1}^{+},\dots, y_{k_b}^{+},z_{j_2}].
\end{align*}
Identity (\ref{6}) implies that $$z_1[y_1^{+},z_2]+z_2[y_1^{+},z_1]-[z_1,y_1^{+},z_2]\in I.$$ If $b>0$ we use  we use this identity to write the commutator $[z_{j_1},y_{1}^{+},\dots, y_{b}^{+},z_{j_2}]$, modulo $I$, as a linear combination of the desired form. This proves the claim.

Therefore $f$ is congruent modulo $I$ to a linear combination of proper polynomials of the form
\begin{align}\label{polyy}
y_{i_1}^{-}\cdots y_{i_a}^{-}[z_{j_1},y_{k_1}^{+},\dots, y_{k_b}^{+}][z_{j_2},y_{t_1}^{+},\dots, y_{t_c}^{+}].
\end{align}

We use identity (\ref{6}) to conclude that \[[z_1,y_1^{+}]z_2+ z_1[z_2,y_1^{+}]\equiv_I0.\] This last equality implies that (\ref{polyy}) is congruent modulo $I$ to a scalar multiple of a polynomial of the form
\begin{align}\label{polyyy}
y_{i_1}^{-}\cdots y_{i_a}^{-}z_{j_1}[z_{j_2},y_{s_1}^{+},\dots, y_{s_r}^{+}].
\end{align}
Moreover if $z_1,z_2$ are both symmetric or skew-symmetric then $z_1[z_2,y_1^{+}]-z_2[z_1,y_1^{+}]$ follows from (\ref{1}) and (\ref{3}), if one is symmetric and the other skew-symmetric then $z_1[y_1^{+},z_2]+z_2[y_1^{+},z_1]$ follows from (\ref{3}) and (\ref{6}). Now we use these identities to order the indices of the odd indeterminates, the identities (\ref{3}) to order the indices in the symmetric/skew-symmetric even indeterminates and conclude $f$ is congruent modulo $I$ to a linear combination of the polynomials in (\ref{basis31}) and (\ref{polyyy}), where the indices of the polynomial in (\ref{polyyy}) satisfy the following conditions: $r>0$, $i_1<\cdots<i_a$, $j_1<j_2$, $s_1<\cdots<s_r$ and $\{i_1,\dots, i_a,j_1,j_2,s_1,\dots, s_r\}=\{1,\dots, n\}$. These polynomials are linearly independent modulo $T_{\mathbb{Z}_2,\ast}(\mathcal{U})$. Therefore we conclude that $f\in I$.
\end{proof}

In the next corollary we compute the codimension sequence for the $(\mathbb{Z}_2,\ast)$-graded identities for $UT_3(F)$ with the grading above.

\begin{corollary}\label{codim3}
	For $\mathcal{U}=UT_3(F)$ with the elementary $\mathbb{Z}_2$-grading induced by $(0,1,0)$ and the reflection involution $\ast$ we have \[c_n^{\mathbb{Z}_2,\ast}(\mathcal{U})=2n(n+5)3^{n-2}-(n-2)2^{n-1}-n,\]
 for all $n\geq 0$.
\end{corollary}
\begin{proof}
It follows from the proof of Theorem \ref{basis} that the images of the polynomials (\ref{y-}), (\ref{basiseven}), (\ref{polyzbas1}), (\ref{polyzbas2}), (\ref{polyzbas3}), (\ref{basis31}) and (\ref{polyyy}) form a basis for $\Gamma_n(\mathcal{U})$. Therefore
\begin{align*}
\gamma_n^{\mathbb{Z}_2,\ast}(\mathcal{U})&=1+n+2n+2n(2^{n-1}-1)+2n\left(2^{n-1}-1\right)+4{n\choose 2}+4{n\choose 2}\left(2^{n-2}-1\right)\\&=1-n+4n2^{n-1}+2n(n-1)2^{n-2},
\end{align*}
for $n\neq 1$. Clearly $\gamma_1^{\mathbb{Z}_2,\ast}(\mathcal{U})=3$.
Proposition \ref{codim} implies that
\begin{align*}
c_n^{\mathbb{Z}_2,\ast}(\mathcal{U})=\sum_{i=0}^{n}{n\choose i}\gamma_i^{\mathbb{Z}_2,\ast}(\mathcal{U}).
\end{align*}
Therefore we have 
\begin{align*}
c_n^{\mathbb{Z}_2,\ast}(\mathcal{U})=\sum_{i=0}^{n}{n\choose i}\left(1-i+4i2^{i-1}+2i(i-1)2^{i-2}\right)-{n\choose 1}.
\end{align*}
The equalities \[i{n\choose i}=n{n-1\choose i-1}\mbox{ \hspace{0,3cm} and \hspace{0,3cm}}i(i-1){n\choose i}=n(n-1){n-2\choose i-2}\] imply that 
\begin{align*}
c_n^{\mathbb{Z}_2,\ast}(\mathcal{U})&=\sum_{i=0}^{n}{n\choose i}-n\left(\sum_{i=1}^{n}{n-1\choose i-1}\right)+4n\left(\sum_{i=1}^{n}{n-1\choose i-1}2^{i-1}\right)+2n(n-1)\left(\sum_{i=2}^{n}{n-2\choose i-2}2^{i-2}\right)-n\\&=2^n-n2^{n-1}+4n3^{n-1}+2n(n-1)3^{n-2}-n,
\end{align*}
this yields the desired equality.
\end{proof}

We remark that for $UT_2(F)$ with a non-trivial elementary $G$-grading the reflection and symplectic involution are graded and an adaptation of the arguments in the proof of Proposition \ref{cod2,3} implies that $c_n^{G,\ast}(UT_2(F))=(n+2)2^{n-1}$. The $\ast$-codimensions for $UT_2(F)$ are given in \cite[Corollary 4.5]{VKS}, we have
\begin{align*}
c_n^{\ast}(UT_2(F))=2^n+n(2^{n-1}-1),
\end{align*}
where $\ast$ is the reflection involution and 
\begin{align*}
c_n^{s}(UT_2(F))=n2^{n-1}+1.
\end{align*}

These remarks together with Corollary \ref{codim3} and Theorem \ref{expo} justify the following conjecture.

\begin{conjecture}
Let $\mathcal{U}$ be a grading by the group $G$ on $UT_m(F)$ and let $\ast$ be a graded involution on $\ast$. There exists $K$, depending on $\mathcal{U}$ and $\ast$, such that \[c_n^{G,\ast}(\mathcal{U})\sim Kn^{m-1}m^{n}.\]
\end{conjecture}

We remark that for the ordinary identities of an affine algebra $W$ there exist $c\in \mathbb{R}$, $t\in \frac{1}{2}\mathbb{Z}$ and an integer $d\geq 0$ such that $c_n(W)\sim cn^td^n$, where $c_n(W)$ is the $n$-th codimension of $W$. The number $t$ is called the polynomial part of the codimension sequence of $W$, an algebraic interpretation for $t$ was given in \cite{AJK}.

\end{document}